\documentclass[11pt, letterpaper]{amsart}

\usepackage[usenames,dvipsnames]{color}
\usepackage{amsthm,amsfonts,amssymb,amsmath,amsxtra}
\usepackage[all]{xy}
\SelectTips{cm}{}

\usepackage{xr-hyper}
\usepackage[
   citecolor=Black,
   linkcolor=Red,
   urlcolor=Blue]{hyperref}
\usepackage{verbatim}

\usepackage{mathrsfs}

\usepackage{cleveref}

\RequirePackage{xspace}
\RequirePackage{etoolbox}
\RequirePackage{varwidth}
\RequirePackage{enumitem}
\RequirePackage{mathtools}
\RequirePackage{longtable}
\RequirePackage{multirow}

\usepackage{tikz-cd}
\usetikzlibrary{arrows}

\def\ge{\geqslant}
\def\le{\leqslant}

\def\p{\pi}

\def\<{\langle}
\def\>{\rangle}

\newcommand{\CA}{\mathcal {A}}
\newcommand{\CB}{\mathcal {B}}
\newcommand{\CC}{\mathcal {C}}

\newcommand{\CO}{\mathcal {O}}

\newcommand{\CQ}{\mathcal {Q}}

\newcommand{\CY}{\mathcal {Y}}

\DeclareMathOperator{\End}{End}

\def\tW{\tilde W}

\newcommand{\tB}[0]{\tilde{B}}
\newcommand{\tU}[0]{\tilde{U}}

\newcommand{\dr}[0]{\dot{r}}

\newcommand{\UPJp}[0]{U_{P_J^+}}
\newcommand{\UPJm}[0]{U_{P_J^-}}

\newtheorem{theorem}{Theorem}
\newtheorem{proposition}[theorem]{Proposition}
\newtheorem{lemma}[theorem]{Lemma}

\newtheorem{corollary}[theorem]{Corollary}

\theoremstyle{definition}
\newtheorem{definition}[theorem]{Definition}
\newtheorem{example}[theorem]{Example}

\newtheorem{remark}[theorem]{Remark}

\numberwithin{equation}{section}
\numberwithin{theorem}{section}

\setitemize[0]{leftmargin=*,itemsep=\the\smallskipamount}
\setenumerate[0]{leftmargin=*,itemsep=\the\smallskipamount}

\renewcommand{\to}{%
   \ifbool{@display}{\longrightarrow}{\rightarrow}%
   }
\let\shortmapsto\mapsto
\renewcommand{\mapsto}{%
   \ifbool{@display}{\longmapsto}{\shortmapsto}%
   }
\newlength{\olen}
\newlength{\ulen}
\newlength{\xlen}
\newcommand{\xra}[2][]{%
   \ifbool{@display}%
      {\settowidth{\olen}{$\overset{#2}{\longrightarrow}$}%
       \settowidth{\ulen}{$\underset{#1}{\longrightarrow}$}%
       \settowidth{\xlen}{$\xrightarrow[#1]{#2}$}%
       \ifdimgreater{\olen}{\xlen}%
          {\underset{#1}{\overset{#2}{\longrightarrow}}}%
          {\ifdimgreater{\ulen}{\xlen}%
             {\underset{#1}{\overset{#2}{\longrightarrow}}}
             {\xrightarrow[#1]{#2}}}}%
      {\xrightarrow[#1]{#2}}
   }
\makeatother
\newcommand{\xyra}[2][]{%
   \settowidth{\xlen}{$\xrightarrow[#1]{#2}$}%
   \ifbool{@display}%
      {\settowidth{\olen}{$\overset{#2}{\longrightarrow}$}%
       \settowidth{\ulen}{$\underset{#1}{\longrightarrow}$}%
       \ifdimgreater{\olen}{\xlen}%
          {\mathrel{\xymatrix@M=.12ex@C=3.2ex{\ar[r]^-{#2}_-{#1} &}}}%
          {\ifdimgreater{\ulen}{\xlen}%
             {\mathrel{\xymatrix@M=.12ex@C=3.2ex{\ar[r]^-{#2}_-{#1} &}}}
             {\mathrel{\xymatrix@M=.12ex@C=\the\xlen{\ar[r]^-{#2}_-{#1} &}}}}}%
      {\mathrel{\xymatrix@M=.12ex@C=\the\xlen{\ar[r]^-{#2}_-{#1} &}}}%
   }
\makeatletter
\newcommand{\xla}[2][]{%
   \ifbool{@display}%
      {\settowidth{\olen}{$\overset{#2}{\longleftarrow}$}%
       \settowidth{\ulen}{$\underset{#1}{\longleftarrow}$}%
       \settowidth{\xlen}{$\xleftarrow[#1]{#2}$}%
       \ifdimgreater{\olen}{\xlen}%
          {\underset{#1}{\overset{#2}{\longleftarrow}}}%
          {\ifdimgreater{\ulen}{\xlen}%
             {\underset{#1}{\overset{#2}{\longleftarrow}}}
             {\xleftarrow[#1]{#2}}}}%
      {\xleftarrow[#1]{#2}}
   }
\newcommand{\isoarrow}{%
   \ifbool{@display}{\overset{\sim}{\longrightarrow}}{\xrightarrow\sim}%
   }

\newif\ifshowrevisions
\showrevisionstrue

\begin{document}

\title{Regularity of totally nonnegative wonderful compactification}


%\date{\today}

\begin{abstract}
Let $G$ be a split real semisimple group of adjoint type.
We prove that the cellular decomposition of the totally nonnegative part of its wonderful
compactification is a regular
CW complex, confirming a conjecture of Galashin, Karp, and Lam.
In particular, the closure of every cell is homeomorphic to a closed ball.
The proof combines a Bruhat atlas compatible with total
positivity and a CL-labeling of the augmented face
poset. The atlas models the strata locally with open
Richardson varieties in the flag variety of a larger Kac--Moody
group and identifies their totally nonnegative cells.
\end{abstract}

\author[Xuhua He]{Xuhua He}
\address{Department of Mathematics and New Cornerstone Science Laboratory, The University of Hong Kong, Pokfulam, Hong Kong, Hong Kong SAR, China}
\email{xuhuahe@hku.hk}
\author{Kaitao Xie}
\address{Department of Mathematics and New Cornerstone Science Laboratory, The University of Hong Kong, Pokfulam, Hong Kong, Hong Kong SAR, China}
\email{kaitaoxie@connect.hku.hk}

\maketitle

%\tableofcontents

\section*{Introduction}\label{section:introduction}
Let $G$ be a split real semisimple group of adjoint type.
The wonderful compactification $\overline G$, introduced by
De Concini and Procesi \cite{DP83}, is a smooth projective
$G\times G$-equivariant compactification of $G$ whose boundary
consists of smooth divisors with normal crossings.
Its $G\times G$-orbits are indexed by subsets of the simple
roots; the unique open orbit is $G$, and the unique closed
orbit is a product of two flag varieties.
Springer's ICM talk \cite{Spr07} surveys its geometry and
connections with representation theory.

In this paper, we show that the wonderful compactification
remains ``wonderful'' in the semi-algebraic setting:
the semi-algebraic cells of its totally nonnegative part
fit together into a regular CW complex.

Let $G_{\ge0}$ be the totally nonnegative part of $G$ in the sense of Lusztig \cite{L94}. The totally nonnegative part $\overline{G_{\ge0}}$ of $\overline{G}$ is defined 
to be the Hausdorff closure of $G_{\ge0}$ in $\overline G(\mathbb R)$.

The first-named author decomposed
$\overline{G_{\ge0}}$ into semi-algebraic cells in \cite{He04}.
Their closure relations and the topology of their closures
remained to be understood. Our main result establishes the
regularity of this cell decomposition, confirming a
conjecture of Galashin--Karp--Lam
\cite{GKL22}.

\begin{theorem}\label{thm:main}
The totally nonnegative wonderful compactification
$\overline{G_{\ge0}}$, equipped with its cell decomposition,
is a regular CW complex. In particular,
$\overline{G_{\ge0}}$ is homeomorphic to a closed ball
of dimension $\dim G$.
\end{theorem}

Regularity is a central theme in total
positivity. For flag varieties, the cell parametrizations
and closure relations were established through the work of
Lusztig, Rietsch, and Marsh--Rietsch
\cite{L94,Ri99,Ri06,MR04}. Galashin--Karp--Lam proved regularity
in finite type \cite{GKL22}, and Bao--He established it for
Kac--Moody flag varieties \cite{BH24}. Regularity for
links of identity in reduced double Bruhat cells was established by Hersh \cite{Her14}, Bao--He \cite{BH22} and the authors of this paper \cite{HX26}. 

In the wonderful compactification,
one must control the interaction between cells in
different $G\times G$-orbits. The proof rests on two main ingredients: a positive Bruhat
atlas for the wonderful compactification and a CL-labeling of its augmented face poset. 

For the geometric construction, we use the stratification obtained by intersecting the $B^+\times B^+$-orbits with the $B^-\times B^-$-orbits in each $G\times G$-orbit \cite{HL11}. In the open orbit, these strata are the double Bruhat cells of $G$; in the closed orbit, they are products of open
Richardson varieties. We show that the intersections of these strata with $\overline{G_{\ge0}}$ are precisely the semi-algebraic cells constructed in \cite{He04}. 

The idea of relating the wonderful compactification to a
Kac--Moody flag variety through an atlas originates in
unpublished work of Knutson, Lu, and the first-named author
\cite{HKL}. Building on this viewpoint, we construct a
different family of atlas maps adapted to total positivity.
The resulting atlas simultaneously models the algebraic
strata and their nonnegative parts, uniformly across all
$G\times G$-orbits.

Let $W$ be the Weyl group of $G$. We construct open charts
$\mathcal O_{(r_1,r_2)}$, indexed by $(r_1,r_2)\in W\times W$,
that cover $\overline G$, together with embeddings
\[
\phi_{(r_1,r_2)}:\mathcal O_{(r_1,r_2)}
\hookrightarrow\tilde{\mathcal B},
\]
where $\tilde{\mathcal B}$ is the flag variety of a larger
Kac--Moody group $\tilde G$. Each map restricts to an isomorphism from every nonempty
intersection of a stratum with the chart onto an open
Richardson variety. Whenever a chart
meets a stratum, it contains  the totally
nonnegative part of this stratum. The atlas map identifies this nonnegative
part with the corresponding totally nonnegative Richardson
cell. We call such
an atlas a \emph{positive Bruhat atlas}.

\begin{theorem}
The wonderful compactification and its totally nonnegative part
admit a positive Bruhat atlas.
\end{theorem}

The atlas transfers the closure
relations and local topological properties of the totally nonnegative flag varieties, established in
\cite{BH24}, to the wonderful compactification.
In particular, the closure of each cell is a topological manifold
with boundary, whose interior is the cell itself.

The combinatorial ingredient is a shellability theorem
for the augmented face poset. Let
$\mathcal Q$ be the face poset of the totally nonnegative cells
in $\overline{G_{\ge0}}$, and let
$\hat{\mathcal Q}=\mathcal Q\sqcup\{\hat0\}$ be obtained
by adjoining a minimum. We establish the following theorem for the
augmented face poset.

\begin{theorem}
The poset $\hat{\mathcal Q}$ is graded, thin, and CL-shellable.
\end{theorem}

The proof begins with an order-reversing embedding
of $\mathcal Q$ into the Weyl group of $\tilde G$ with convex
image, providing a combinatorial counterpart to the geometric
atlas. This embedding identifies intervals in $\mathcal Q$
with dual Bruhat intervals.
The essential step is to construct a labeling on the entire augmented poset, including the intervals containing the adjoined minimum.
These intervals record the incidence relations among the
cells in the closure of each cell, and shellability allows us to determine the topology of its boundary.

The positive atlas endows the closure of each cell a manifold structure and identifies its boundary. The combinatorial result determines
the topology of the boundary inductively. The topological criterion recalled in \S\ref{section:regularity}, based on the generalized
Poincar\'e conjecture, then identifies the closure of each cell with a closed ball and proves \Cref{thm:main}.

The paper is organized as follows. In
\S\ref{section:prelimanary} we recall the wonderful
compactification and its totally nonnegative part.
In \S\ref{section:bruhat_atlas} we construct the combinatorial
and geometric atlases. In \S\ref{section:positivity} we prove
compatibility with total positivity and deduce the closure
relations and the manifold structure for the closure of each cell. In \S\ref{section:combinatorics} we prove thinness and
CL-shellability of the augmented face poset. We combine these
results in \S\ref{section:regularity} to prove the regularity
theorem. Finally, \S\ref{section:examples} gives explicit
examples in types $A_1$ and $A_2$.
\subsection*{Acknowledgment}
XH is partially supported by the New Cornerstone Science Foundation through the New Cornerstone Investigator Program and the Xplorer Prize, as well as by the Hong Kong RGC grant 14300023. KX is partially supported by NSFC grant 125B2002 and the New Cornerstone Science Foundation through the New Cornerstone Investigator Program awarded to XH. We thank Huanchen Bao for many useful discussions during an early stage of this project. This manuscript was prepared with the assistance of AI tools to improve linguistic readability.

\section{The wonderful compactification}\label{section:prelimanary}

\subsection{Reductive groups}\label{sec:semisimple_group}

Let $G$ be a connected reductive group split over $\mathbb R$. Fix a maximal torus $T$, and a pair of opposite Borel subgroups $B^+$ and $B^-$ such that $B^+\cap B^- = T$. Denote the unipotent radical of $B^\pm$ by $U^\pm$. Let $\Phi$ be the root system of $G$. For each $\alpha\in\Phi$, we write $U_\alpha$ for the corresponding root subgroup. 

Let $I$ be the index set of simple roots. For $i\in I$, we write $\alpha_i$ the corresponding simple root, and $\alpha^\vee_i$ the corresponding simple coroot. Let $W$ be the Weyl group of $G$. We write $\ell$ for its length function, $\leq$ for the Bruhat order, and $w_0$ for its longest element.

For $J\subseteq I$, let $P_J^+$ be the {\it parabolic subgroup} of $G$ generated by $B^+$ and $U_{-\alpha_i}$ for $i\in J$. Let $P_J^-$ be the {\it opposite parabolic subgroup} of $G$ generated by $B^-$ and the root groups $U_{\alpha_i}$ for $i\in J$. Let $L_J = P_J^+\cap P_J^-$, and $U_{P_J^\pm}$ be the unipotent radical of $P_J^\pm$. We have the Levi decomposition $P_J^\pm = L_J\ltimes U_{P_J^\pm}$. We denote $B_J^\pm:=L_J\cap B^\pm$ and $U_J^\pm:=L_J\cap U^\pm$.

For $J\subseteq I$, let $W_J$ be the parabolic subgroup of $W$ generated by $s_j$ for $j\in J$. Let $W^J$ denote the set of minimal representatives for the cosets in $W/W_J$. Every element $w\in W$ has a unique decomposition $w = w^J w_J$ with $w^J\in W^J$ and $w_J\in W_J$.

\subsection{Total positivity} \label{sec:prelim_positivity}
 For every $i\in I$, we fix isomorphisms $x_i:\mathbb R\to U_{\alpha_i}$, $y_i:\mathbb R\to U_{-\alpha_i}$ such that the map
$$\begin{pmatrix} 1 & a \\ 0 & 1 \end{pmatrix}\mapsto x_i(a), \begin{pmatrix} 1 & 0 \\ b & 1 \end{pmatrix}\mapsto y_i(b),
\begin{pmatrix} c & 0 \\ 0 & c^{-1} \end{pmatrix}\mapsto \alpha^\vee_i(c), a,b\in\mathbb R, c\in\mathbb R^*,$$
defines a homomorphism $SL_2\to G$. The datum $(T,B^+,B^-,x_i,y_i;i\in I)$ is called a {\it pinning} for $G$.

Let $U^+_{\ge0}$ be the submonoid of $G$ generated by $x_i(a)$ for $i\in I$ and $a\in\mathbb R_{>0}$, and $U^-_{\ge0}$ be the submonoid of $G$ generated by $y_i(a)$ for $i\in I$ and $a\in\mathbb R_{>0}$. Let $T_{>0}$ be the identity component of $T(\mathbb R)$. The {\it totally nonnegative monoid} $G_{\ge0}$ is defined to be the submonoid of $G$ generated by $U^\pm_{\ge 0}$ and $T_{>0}$. It is known that $$G_{\ge0} = U^+_{\ge0}T_{>0}U^-_{\ge0} = U^-_{\ge0}T_{>0}U^+_{\ge0}.$$

There is an anti-involution $\iota:G\to G$ determined by
\[
\iota(x_i(a))=y_i(a),\enspace
\iota(y_i(a))=x_i(a),\enspace \iota(t)=t
\]
for $a\in\mathbb R$ and $t\in T$. It preserves $G_{\ge0}$.

For $i \in I$, we define $\dot{s}_i = x_i(-1) y_i(1) x_i(-1) \in G$. For any $w \in W$ with a reduced expression $w = s_{i_1} \cdots s_{i_n}$, we set $\dot{w} = \dot{s}_{i_1} \cdots \dot{s}_{i_n} \in G$. The element $\dot{w}$ is independent of the choice of reduced expression. 

For $J\subseteq I$, let $L_{J,\ge0}$ be the totally nonnegative part of $L_J$ defined from the restriction of the pinning of $G$. Then $L_{J,\ge0} = L_J\cap G_{\ge0}$.

For $w\in W$, choose a reduced expression
$w=s_{i_1}\cdots s_{i_n}$ and set
\[
\begin{aligned}
U^-_{w,>0}
&=\{y_{i_1}(a_1)\cdots y_{i_n}(a_n)|
       a_1,\ldots,a_n>0\},\\
U^+_{w,>0}
&=\{x_{i_1}(a_1)\cdots x_{i_n}(a_n)|
       a_1,\ldots,a_n>0\}.
\end{aligned}
\]
By \cite[Lemma~2.3]{L94}, these subsets are independent of the reduced expression. Moreover, Lusztig \cite[Proposition~2.7 \& Corollary~2.8]{L94} showed that for $v,w\in W$, $G_{\ge0}\cap B^+\dot wB^+\cap B^-\dot vB^-
   =U^-_{w,>0}T_{>0}U^+_{v,>0}$.
In particular, we have the totally positive part $G_{>0}:=U^-_{w_0,>0}T_{>0}U^+_{w_0,>0} = U_{w_0,>0}^+T_{>0}U_{w_0,>0}^-$.

\subsection{The wonderful compactification}\label{sec:prelim_wonderful}
In this subsection, we assume that $G$ is of adjoint type. The {\it wonderful compactification} of $G$, denoted by $\overline G$, was introduced by De Concini and Procesi \cite{DP83}. It is a smooth projective $G \times G$-equivariant completion of $G$ and the boundary is a union of smooth divisors with normal crossings. We review some basic properties of $\overline G$, and refer the reader to \cite{DP83,He05,Spr07,HL11} for comprehensive discussions.

Let $G_{\text{sc}}$ be the simply connected cover of $G$. Let $\lambda$ be a regular dominant weight of $G_{\text{sc}}$ and $V_\lambda$ be the highest weight representation associated to $\lambda$. The action of $G_{\text{sc}}$ on $V_\lambda$ gives an embedding $G\hookrightarrow\mathbb P(\End(V_\lambda))$ since $G$ is of adjoint type. The Zariski closure of the image of $G$ in $\mathbb P(\End(V_\lambda))$ is isomorphic to $\overline G$.

The variety $G$ carries a $G\times G$-action by $(g_1,g_2).g = g_1gg_2^{-1}$. This action extends to a $G\times G$-action on $\overline G$. For $J\subseteq I$, we choose the standard point $z_J\in\overline G$, such that the stabilizer of $z_J$ in $G\times G$ is $$\text{Stab}_{G\times G}(z_J) = (U_{P_J^+},U_{P_J^-})\{(l_1,l_2)\in L_J\times L_J|l_1l_2^{-1}\in Z(L_J)\}.$$
Here, $Z(L_J)$ is the center of $L_J$. We have $\overline G = \bigsqcup_{J\subseteq I}(G\times G)\cdot z_J$. We shall write an element in $\overline G$ as $(g_1,g_2)\cdot z_J$.

The \emph{totally nonnegative part} of the wonderful compactification is $\overline{G_{\ge0}}\subseteq\overline G(\mathbb R)$,
where the closure is taken in the Hausdorff topology. 
The totally positive part $G_{>0}$ is dense in $G_{\ge0}$. Consequently, we have $\overline{G_{>0}}=\overline{G_{\ge0}}$.

\begin{example}\label{ex:rank_one}

We have the embedding
$$PGL_2\hookrightarrow\mathbb P^3,\quad \begin{pmatrix}
    a &b \\
    c &d
\end{pmatrix}\mapsto [a:b:c:d]$$
and we identify $PGL_2$ with 
$\left\{[a:b:c:d]\in\mathbb P^3\big|ad-bc\neq0\right\}.$
Then $\overline{PGL_2}\simeq\mathbb P^3$. It follows from a result of Whitney \cite{Whit52} that the totally nonnegative part of the wonderful compactification is
$$\overline{PGL_{2,\ge0}} = \left\{[a:b:c:d]\in\mathbb P^3\big|a,b,c,d,ad-bc\ge0\right\}.$$
We return to the details of this example in \S\ref{sec:rank_one_details}.
\end{example}

\subsection{A stratification on the wonderful compactification}\label{sec:wonderful_strata}

We consider the collection of triples
$$\CA = \{(J;x,y)|J\subseteq I,x\in W^J,y\in W\}.$$
For $\alpha = (J;x,y)\in\CA$, we denote
$$[\alpha]_+ = (B^+,B^+)(\dot x,\dot y)\cdot z_J,\ [\alpha]_- = (B^-,B^-)(\dot x,\dot y)\cdot z_J.$$
We define partial orders $\leq_+$ and $\leq_-$ on $\CA$ as follows. Suppose that $\alpha = (J;x,y)$ and $\alpha' = (J';x',y')$, we have $\alpha'\leq_+\alpha$ if $J'\subseteq J$ and there exists $u\in W_J$ such that $x'w_{J',0}\leq xu$ and $yu\leq y'w_{J',0}$, where $w_{J',0}$ is the longest element in $W_{J'}$; we have $\alpha'\leq_-\alpha$ if $J'\subseteq J$ and there exists $u\in W_J$ such that $xu\leq x'$ and $y'\leq yu$.

According to \cite[\S3.3]{HL11}, we have
\begin{enumerate}
    \item $\overline G = \bigsqcup_{\alpha\in\CA}[\alpha]_+ = \bigsqcup_{\alpha\in\CA}[\alpha]_- $ are two stratifications;
    \item the Zariski closures $\overline{[\alpha]_+} = \bigsqcup_{\alpha'\leq_+\alpha}[\alpha']_+$ and $\overline{[\alpha]_-} = \bigsqcup_{\alpha'\leq_-\alpha}[\alpha']_-$.
\end{enumerate}

Suppose that $\alpha = (J_1;x,y)$ and $\beta = (J_2;v,w)$ with $J_1=J_2 = J$, we define the concatenation of $\alpha$ and $\beta$ by $\alpha\diamond\beta = (J;x,y;v,w)$, and let
$$[\alpha\diamond\beta] = [\alpha]_+\bigcap[\beta]_-.$$
It follows from \cite[Proposition~3.1 \& Theorem~3.2]{HL11} that
\begin{enumerate}
    \item for $\alpha\diamond\beta = (J;x,y;v,w)$, we have $[\alpha\diamond\beta]\neq\emptyset$ if and only if $v\leq x$ and $y^J\leq w^J$;
    \item $\overline G = \bigsqcup [\alpha\diamond\beta]$ is a stratification;
    \item the Zariski closure $\overline{[\alpha\diamond\beta]} = \bigsqcup_{\alpha'\leq_+\alpha,\beta'\leq_-\beta} [\alpha'\diamond\beta']$.
    \item if $[J;x,y;v,w]\neq\emptyset$, then $$\dim [J;x,y;v,w]=|J|+\ell(x)+\ell(w)+\ell(y_J)-\ell(v)-\ell(y^J).$$
\end{enumerate}
Let $\mathcal{Q} = \{ \alpha \diamond \beta \mid [\alpha \diamond \beta] \neq \emptyset \}$, and define a partial order $\le_{\mathcal{Q}}$ on $\mathcal{Q}$ by $\alpha' \diamond \beta' \le_{\mathcal{Q}} \alpha \diamond \beta$ if $\alpha' \le_+ \alpha$ and $\beta' \le_- \beta$. For $\alpha\diamond\beta\in\CQ$, set $[\alpha\diamond\beta]_{>0} = \overline{G_{\ge0}}\bigcap [\alpha\diamond\beta]$. Then we have $\overline{G_{\ge0}} = \bigsqcup_{\alpha\diamond\beta\in\CQ} [\alpha\diamond\beta]_{>0}$.

Note that minimal elements of the poset $\CQ$ are $(\emptyset;r_1,r_2;r_1,r_2)$ for $r_1,r_2\in W$. We denote $\xi_{(r_1,r_2)}:=(\emptyset;r_1,r_2;r_1,r_2)$. The following description follows directly from the definition of the partial order. 

\begin{lemma} \label{atoms}
Let $r_1,r_2\in W$ and $q=(J;x,y;v,w)\in\CQ$. We have $\xi_{(r_1,r_2)}\leq_{\mathcal Q}q$ if and only if there exist $z_1,z_2\in W_J$, such that
$vz_2\leq r_1\leq xz_1$, and $yz_1\leq r_2\leq wz_2$. In particular, we have $\xi_{(v,y^J)}\leq q$, and if $\xi_{(r_1,r_2)}\leq q$, then $v\leq r_1$, $y^J\leq r_2$.
\end{lemma}

\section{A Bruhat atlas for the wonderful compactification}\label{section:bruhat_atlas}

\subsection{Kac--Moody groups and flag varieties}
Let $I$ be a finite set, and let $A = (a_{ij})_{i,j\in I}$ be a symmetrizable generalized Cartan matrix. A {\it Kac--Moody datum} associated with $A$ is a  sextuple $\mathcal D = (I,A,X,Y,(\alpha_i)_{i\in I}, (\alpha^\vee_i)_{i\in I})$ where $X$ and $Y$ are two free $\mathbb Z$-modules dual to each other, and the simple roots $\alpha_i\in X$ and simple coroots $\alpha_j^\vee\in Y$ satisfy $\langle\alpha_i,\alpha^\vee_j\rangle = a_{ij}$ for all $i,j\in I$. The {\it minimal Kac--Moody group} $\tilde G$ over $\mathbb C$ associated with $\mathcal D$ is the split group over $\mathbb C$ generated by the split torus $T$ associated with $Y$ and the root groups $U_{\pm\alpha_i}$ for all $i\in I$, subject to the Tits relations. 

The notions of Borel subgroups, parabolic subgroups, Levi subgroups, etc. in \S\ref{sec:semisimple_group} generalize to the Kac--Moody setting. We use $\tilde B^\pm$, $\tilde P_J^\pm$, $\tilde L_J$ etc. to distinguish the Kac--Moody objects from finite type objects whenever there might be a confusion.

Let $\tilde{\mathcal B}=\tilde G/\tilde B^+$ be the flag variety of $\tilde G$.
For $\tilde v,\tilde w\in\tilde W$, write
\[
\mathring{\tilde{\mathcal B}}_{\tilde w}=\tilde B^+\dot{\tilde w}\tilde B^+/\tilde B^+,
\qquad
\mathring{\tilde{\mathcal B}}^{\tilde v}=\tilde B^-\dot{\tilde v}\tilde B^+/\tilde B^+
\]
for the Schubert and opposite Schubert cells, respectively.
Their intersection
\[
\mathring{\tilde{\mathcal B}}_{\tilde v,\tilde w}=\mathring{\tilde{\mathcal B}}_{\tilde w}\cap\mathring{\tilde{\mathcal B}}^{\tilde v}
\]
is the \emph{open Richardson variety}. It is nonempty if and
only if $\tilde v\leq\tilde w$. These varieties give the stratifications
\[
\tilde{\mathcal B}
 =\bigsqcup_{\tilde w\in\tilde W}\mathring{\tilde\CB}_{\tilde w}
 =\bigsqcup_{\tilde v\in\tilde W} \mathring{\tilde\CB}^{\tilde v}
 =\bigsqcup_{\tilde v\leq \tilde w}\mathring{\tilde\CB}_{\tilde v,\tilde w}.
\]

Denote their Zariski closures by
$\tilde{\mathcal B}_{\tilde w}$,
$\tilde{\mathcal B}^{\tilde v}$, and
$\tilde{\mathcal B}_{\tilde v,\tilde w}$, respectively.
We have 
\[
\tilde{\mathcal B}_{\tilde w}
 =\bigsqcup_{\tilde w'\leq \tilde w}\mathring{\tilde\CB}_{\tilde w'}, \quad
\tilde{\mathcal B}^{\tilde v}
 =\bigsqcup_{\tilde v'\geq\tilde v}\mathring{\tilde\CB}^{\tilde v'}, \quad
\tilde{\mathcal B}_{\tilde v,\tilde w}
 =\bigsqcup_{\tilde v\leq \tilde v'\leq\tilde w'\leq\tilde w}\mathring{\tilde{\mathcal B}}_{\tilde v',\tilde w'}.
\]

\subsection{Bruhat atlas}

Let $M$ be a finite dimensional irreducible variety with a stratification $M = \bigsqcup_{y \in \CY} \mathring{M}_y$ by finitely many strata. A {\it Bruhat atlas} on $(M, \CY)$ consists of the following data:

\begin{itemize}
    \item A Kac--Moody group $\tilde G$
    
    \item an open covering $M=\cup_{f\in F} U_f$;

     \item for each $f\in F$, an embedding $\phi_f:U_f\to\tilde\CB$, where $\tilde\CB$ is the flag variety of $\tilde G$ such that whenever $\mathring M_y \cap U_f \neq \emptyset$, the image $\phi_f(U_f \cap \mathring M_y)$ is an open Richardson variety in $\tilde\CB$.
\end{itemize}

The main result of this section is the construction of such an atlas for the wonderful compactification $\overline G$.

\subsection{The atlas group}
\label{sec:atlas_group}

We construct a larger Kac--Moody group whose flag variety will
provide the target of our atlas maps. Its Weyl group will also
provide a combinatorial model for the closure poset $\mathcal Q$.

Let $I$ be the index set of simple roots of $G$, and $A = (a_{ij})$ its Cartan matrix. Take three copies $I^\flat, I^\natural, I^\sharp$ of $I$ and set
\[
\tilde I = I^\flat \sqcup I^\natural \sqcup I^\sharp.
\]
Define a generalized Cartan matrix $\tilde A = (\tilde a_{\tilde i,\tilde j})$ by:
\begin{itemize}
    \item for $i,j\in I$: $\tilde a_{i^\flat,j^\flat} = \tilde a_{i^\sharp,j^\sharp} = a_{ij}$;
    \item for $\tilde i\in I^\flat$, $\tilde j\in I^\sharp$: $\tilde a_{\tilde i,\tilde j} = \tilde a_{\tilde j,\tilde i} = 0$;
    \item for $i\in I$: $\tilde a_{i^\flat,i^\natural} = \tilde a_{i^\natural,i^\flat} = -1$, $\tilde a_{i^\sharp,i^\natural} = \tilde a_{i^\natural,i^\sharp} = -1$;
    \item for $i\neq j$: $\tilde a_{i^\flat,j^\natural} = \tilde a_{j^\natural,i^\flat} = 0$, $\tilde a_{i^\sharp,j^\natural} = \tilde a_{j^\natural,i^\sharp} = 0$, $\tilde a_{i^\natural,j^\natural} = \tilde a_{j^\natural,i^\natural} = 0$;
    \item for all $\tilde i\in\tilde I$: $\tilde a_{\tilde i,\tilde i} = 2$.
\end{itemize}
Let $\tilde G$ be the minimal Kac–Moody group associated to $(\tilde I,\tilde A)$. 

\begin{example}
When the Coxeter system $(W,I)$ is of type $F_4$, the Coxeter diagram of $\tilde W$ is as follows:
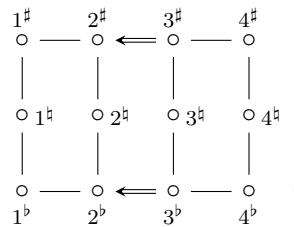
\begin{center}
\begin{tikzpicture}[baseline=0]
	\node  (a) at (0,0) {$\circ$};
	\node   at (0,-0.3) {$\scriptstyle 1^\flat$};
	\node (b) at (1,0) {$\circ$};
	\node at (1, -0.3) {$\scriptstyle 2^\flat$};
	\node (c) at (2,0) {$\circ$};
	\node at (2, -0.3) {$\scriptstyle 3^\flat$};
    \node (d) at (3,0) {$\circ$};
    \node at (3,-0.3) {$\scriptstyle 4^\flat$};
    
    \node (e) at (0,1) {$\circ$};
    \node at (0.3,1) {$\scriptstyle 1^\natural$};
    \node (f) at (1,1) {$\circ$};
    \node at (1.3,1) {$\scriptstyle 2^\natural$};
    \node (g) at (2,1) {$\circ$};
    \node at (2.3,1) {$\scriptstyle 3^\natural$};
    \node (h) at (3,1) {$\circ$};
    \node at (3.3,1) {$\scriptstyle 4^\natural$};

    \node (i) at (0,2) {$\circ$};
    \node at (0,2.3) {$\scriptstyle 1^\sharp$};
    \node (j) at (1,2) {$\circ$};
    \node at (1,2.3) {$\scriptstyle 2^\sharp$};
    \node (k) at (2,2) {$\circ$};
    \node at (2,2.3) {$\scriptstyle 3^\sharp$};
    \node (l) at (3,2) {$\circ$};
    \node at (3,2.3){$\scriptstyle 4^\sharp$};
    
	\draw (a) -- (b) ;
    \draw[>=stealth, double, double distance=1pt, <-] (b) -- (c);
    \draw (c) -- (d);
    \draw (i) -- (j);
    \draw[>=stealth, double, double distance=1pt, <-] (j) -- (k);
    \draw (k) -- (l);
    \draw (a) -- (e) --  (i);
    \draw (b) -- (f) --  (j);
    \draw (c) --  (g) --  (k);
    \draw (d) -- (h) -- (l);
	\end{tikzpicture}.
\end{center}
\end{example}

Denote by $\tilde L_{I^\flat}$ and $\tilde L_{I^\sharp}$ the standard Levi subgroups corresponding to $I^\flat$ and $I^\sharp$. By definition, both $\tilde L_{I^\sharp}/Z(\tilde L_{I^\sharp})$ and $\tilde L_{I^\flat}/Z(\tilde L_{I^\flat})$ are isomorphic to $G$. These identifications induce isomorphisms 
$$U^\pm\xrightarrow[]{\sim}\tilde U_{I^\sharp}^\pm, u\mapsto u^\sharp,\text{ and }U^\pm\xrightarrow[]{\sim}\tilde U_{I^\flat}^\pm, u\mapsto u^\flat,$$

We choose a pinning $(\tilde T,\tilde B^+,\tilde B^-,\tilde x_{\tilde i},\tilde y_{\tilde i};\tilde i\in\tilde I)$ of $\tilde G$ so that
$$x_i(a)^\sharp = \tilde x_{i^\sharp}(a),\enspace y_i(a)^\sharp = \tilde y_{i^\sharp}(a),\enspace x_i(a)^\flat =\tilde x_{i^\flat}(a),\enspace y_i(a)^\flat = \tilde y_{i^\flat}(a).$$
For each $\tilde i\in\tilde I$, we have the notations $\tilde\alpha^\vee_{\tilde i}$ and $\dot{\tilde s}_{\tilde i}$ for $\tilde G$ as in \S\ref{sec:prelim_positivity}. This also induces the following notations
 $$   \dot s_i^\sharp = x_i(-1)^\sharp y_i(1)^\sharp x_i(-1)^\sharp,\enspace \dot s_i^\flat = x_i(-1)^\flat y_i(1)^\flat x_i(-1)^\flat, $$
 $$\alpha_i^\vee(a)^\sharp = x_i(a)^\sharp y_i(-a^{-1})^\sharp x_i(a)^\sharp\dot s_i^\sharp,\enspace \alpha^\vee_i(a)^\flat = x_i(a)^\flat y_i(-a^{-1})^\flat x_i(a)^\flat\dot s_i^\flat$$
 for each $i\in I$. 

For $t\in T\subseteq G$, we write $t = \prod_{i\in I}\alpha_i^\vee(a_i)$, and denote
$$t^\sharp = \prod_{i\in I}\alpha^\vee_i(a_i)^\sharp,\enspace t^\flat = \prod_{i\in I}\alpha^\vee_i(a_i)^\flat.$$
However, since $G$ is assumed to be of adjoint type, the product expression of $t$ is not unique. Therefore $t^\sharp$ (resp. $t^\flat$) is only defined up to a central element in the derived subgroup of $\tilde L_{I^\sharp}$ (resp. $\tilde L_I^\flat$).

\subsection{Combinatorial atlas}\label{sec:combinatorial_atlas}

Let $\tilde W$ be the Weyl group of $\tilde G$, considered as a poset with the usual Bruhat order. In this section, we embed $\CQ$ into $\tilde W$. We have the standard identifications: \[ W\simeq\tilde W_{I^\flat},\qquad W\simeq\tilde W_{I^\sharp}. \] For $w\in W$, denote its images by $w^\flat$ and $w^\sharp$.

For $K\subseteq I$, let $K^\natural = \{k^\natural \mid k\in K\}$ and $c_{K^\natural} = \prod_{k\in K} s_{k^\natural}$ be the Coxeter element in $\tilde W_{K^\natural}$. Let $w_0$ be the longest element of $W$.

Consider $\mathcal A\diamond\mathcal A := \{\alpha\diamond\beta \mid \alpha,\beta\in\mathcal A\}$ with the partial order $\leq_{\mathcal A\diamond\mathcal A}$ given by $\alpha'\diamond\beta' \leq_{\mathcal A\diamond\mathcal A} \alpha\diamond\beta$ if $\alpha'\leq_+\alpha$ and $\beta'\leq_-\beta$. Define
\begin{gather*}
\varphi_+ : \mathcal A \to \tilde W,\qquad (J;x,y) \mapsto y^\sharp \, c_{(I\setminus J)^\natural} \, (x^{-1}w_0)^\sharp, \\
\varphi_- : \mathcal A \to \tilde W,\qquad (J;x,y) \mapsto x^\flat \, c_{(I\setminus J)^\natural} \, (y^{-1}w_0)^\flat,\\
\varphi : \mathcal A\diamond\mathcal A \to \tilde W,\qquad (J;x,y;v,w) \mapsto v^\flat \, y^\sharp \, c_{(I\setminus J)^\natural} \, (x^{-1}w_0)^\sharp \, (w^{-1}w_0)^\flat.
\end{gather*}

According to \cite[Theorem 2.3]{CD03}, we have:
\begin{itemize}
    \item $\varphi_+$ defines an order-reversing embedding from $(\mathcal A,\leq_+)$ into $(\tilde W,\leq)$ with convex image $\bigcup_{K\subseteq I} \tilde W_{I^\sharp} c_{K^\natural} \tilde W_{I^\sharp}$;
    \item $\varphi_-$ defines an order-reversing embedding from $(\mathcal A,\leq_-)$ into $(\tilde W,\leq)$ with convex image $\bigcup_{K\subseteq I} \tilde W_{I^\flat} c_{K^\natural} \tilde W_{I^\flat}$.
\end{itemize}
It then follows that $\varphi$ is an order-reversing embedding from $(\mathcal A\diamond\mathcal A,\leq_{\mathcal A\diamond\mathcal A})$ into $(\tilde W,\leq)$ with convex image $\bigcup_{K\subseteq I} \tilde W_{I^\flat\sqcup I^\sharp} c_{K^\natural} \tilde W_{I^\flat\sqcup I^\sharp}$.

\begin{proposition} \label{combinatorial_atlas}
The restriction of $\varphi$ to $\CQ$ is an order-reversing embedding with convex image
$$\{\tilde w\in\tilde W\big|\tilde w\leq r_1^\flat r_2^\sharp c_{I^\natural}(r_1^{-1}w_0)^\sharp (r_2^{-1}w_0)^\flat\text{ for some }r_1,r_2\in W\}.$$

\begin{proof}
Denote $$\CC := \{\tilde w\in\tilde W\big|\tilde w\leq r_1^\flat r_2^\sharp c_{I^\natural}(r_1^{-1}w_0)^\sharp (r_2^{-1}w_0)^\flat,\text{ for some }r_1,r_2\in W\}.$$
Note that $\CC$ is a union of some Bruhat lower ideals and thus is convex. It is sufficient to show that $\varphi(\CQ) = \CC$. 

Suppose that $(J;x,y;v,w)\in\CQ$, we have $y^J\leq w^J$ and $v\leq x$. There exist $r_1,r_2\in W$ such that $v\leq r_1\leq x\leq xy_J^{-1}$ and $y^J\leq r_2\leq w^J\leq w$. Note that $c_{(I\setminus J)^\natural}$ commutes with $r^\flat$ and $r^\sharp$ for any $r\in W_J$. We may write
$$\varphi(J;x,y;v,w) = v^\flat (y^J)^\sharp c_{(I\setminus J)^\natural}(y_Jx^{-1}w_0)^\sharp (w^{-1}w_0)^\flat,$$
and we have
$\varphi(J;x,y;v,w)\leq r_1^\flat r_2^\sharp c_{I^\natural}(r_1^{-1}w_0)^\sharp (r_2^{-1}w_0)^\flat$. Thus, $\varphi(\CQ)\subseteq\CC$.

Conversely, for $\tilde w\in\CC$, we see that there exist $(J;x,y;v,w)\in\CA\diamond\CA$ such that $\varphi(J;x,y;v,w) = \tilde w$. Suppose that $\varphi(J;x,y;v,w)\leq
r_1^\flat r_2^\sharp c_{I^\natural}
(r_1^{-1}w_0)^\sharp(r_2^{-1}w_0)^\flat$. Then $v\leq r_1^J\leq x$ and $y^J\leq r_2^J\leq w^J$. We have $y^J\leq w^J$ and $v\leq x$. Therefore, $(J;x,y;v,w)\in\CQ$. We conclude that $\varphi(\CQ) = \CC$.
\end{proof}
\end{proposition}

\subsection{Open charts and atlas maps}\label{sec:atlas_map}

Set $\overline T^\circ := \bigsqcup_{J\subseteq I} (T,T)\cdot z_J$. For $i\in I$, let $\omega_i^\vee$ be the corresponding fundamental coweight of $G$. There is an isomorphism $\pi:\mathbb C^I \to \overline T^\circ$ given by
\[
\underline{a} = (a_i)_{i\in I} \;\longmapsto\; \Bigl(\prod_{i: a_i\neq0} \omega_i^\vee(a_i^{-1}),\; e\Bigr) \cdot z_{\{i\mid a_i\neq0\}}.
\]
Write $t_{\underline{a}} = \prod_{i\in I,a_i\neq0}\omega_i^\vee(a_i^{-1})$. For $r_1,r_2\in W$, define the open chart
\[
\mathcal O_{(r_1,r_2)} = (\dot r_1 U^-, \dot r_2 U^+) \cdot \overline T^\circ \subseteq \overline G.
\]
We have an identification
\[
\begin{aligned}
(\dot r_1U^-\dot r_1^{-1})\times
(\dot r_2U^+\dot r_2^{-1})\times\mathbb C^I
&\xrightarrow{\sim}\mathcal O_{(r_1,r_2)},\\
(u_1,u_2;\underline a)
&\longmapsto
(u_1\dot r_1,u_2\dot r_2)
\cdot\pi(\underline{a})
\end{aligned}
\]

These charts cover $\overline G$. Indeed, their union is a $T\times T$-stable open subset containing all the $T\times T$-fixed points $(\dot r_1,\dot r_2)\cdot z_\emptyset$. These are all the torus fixed points of $\overline G$. If the complement were nonempty, its projectivity and $T\times T$-stability would imply that it contains a torus fixed point, a contradiction.

First recall the following isomorphisms for $r\in W$:
\[
(\dot r U^- \dot r^{-1} \cap U^+) \times (\dot r U^- \dot r^{-1} \cap U^-) \xrightarrow{\sim} \dot r U^- \dot r^{-1},\quad (g_1,g_2)\mapsto g_1g_2,
\]
\[(\dot r U^- \dot r^{-1} \cap U^-) \times (\dot r U^- \dot r^{-1} \cap U^+) \xrightarrow{\sim} \dot r U^- \dot r^{-1},\quad (g_1',g_2')\mapsto g_1'g_2',\]
and analogously for $\dot r U^+ \dot r^{-1}$. Define the projections
\begin{gather*}
\sigma_{r,-} : \dot r U^- \dot r^{-1} \to \dot r U^- \dot r^{-1} \cap U^-,\quad g_1g_2 \mapsto g_2, \\
\sigma_{r,+} : \dot r U^- \dot r^{-1} \to \dot r U^- \dot r^{-1} \cap U^+,\quad g_1'g_2' \mapsto g_2',\\
\tau_{r,-} : \dot r U^+ \dot r^{-1} \to \dot r U^+ \dot r^{-1} \cap U^-,\quad h_1h_2 \mapsto h_2,\\
\tau_{r,+} : \dot r U^+ \dot r^{-1} \to \dot r U^+ \dot r^{-1} \cap U^+,\quad h_1'h_2' \mapsto h_2'.
\end{gather*}

Let 
\[
m : \mathbb C^I \to \tilde L_{I^\natural},\qquad
\underline{a} \mapsto \Bigl(\prod_{i: a_i\neq0} \tilde y_{i^\natural}(a_i^{-1})\tilde \alpha^\vee_{i^\natural}(a_i)\Bigr) \Bigl(\prod_{j: a_j=0} \dot{\tilde s}_{j^\natural}\Bigr),
\]
and let $w_0$ be the longest element of $W$. 

We define the {\it atlas map} $\phi_{(r_1,r_2)} : \mathcal O_{(r_1,r_2)} \longrightarrow \tilde{\mathcal B}$ by sending $(u_1,u_2;\underline{a})$ to
\[\begin{aligned}&\bigl(\iota(\tau_{r_2,-}(u_2)\dot r_2)\bigr)^{-1,\sharp}\; \bigl(\sigma_{r_1,+}(u_1)\dot r_1 \bigr)^\flat\:m(\underline{a})\\ &\quad  \bigl(\tau_{r_2,+}(u_2)\dot r_2\bigr)^{-1,\flat}\; \bigl(\iota(\sigma_{r_1,-}(u_1)\dot r_1 )\bigr)^\sharp\;\dot w_0^\sharp \dot w_0^\flat\tilde B^+/\tilde B^+.\end{aligned}\]
Note that the derived subgroups of $\tilde L_{I^\sharp}$ and $\tilde L_{I^\flat}$ commute with each other.

Denote 
\begin{align*}
 \phi_{(r_1,r_2)}^\sharp : \mathcal O_{(r_1,r_2)} &\to \tilde{\mathcal B}, \\   
 (u_1,u_2;\underline{a}) &\mapsto \bigl(\iota(\tau_{r_2,-}(u_2)\dot r_2)\bigr)^{-1,\sharp}\; m(\underline{a})\; \bigl(\iota(\sigma_{r_1,-}(u_1)\dot r_1 )\bigr)^\sharp\; \dot w_0^\sharp\tilde B^+/\tilde B^+; \\
 \phi_{(r_1,r_2)}^\flat : \mathcal O_{(r_1,r_2)} &\to \tilde{\mathcal B}, \\
 (u_1,u_2;\underline{a}) &\mapsto \bigl(\sigma_{r_1,+}(u_1)\dot r_1 \bigr)^\flat\; m(\underline{a})\; \bigl(\tau_{r_2,+}(u_2)\dot r_2\bigr)^{-1,\flat}\; \dot w_0^\flat\tilde B^+/\tilde B^+
\end{align*}
We often use a concatenation shorthand: $$\phi_{(r_1,r_2)}(u_1,u_2;\underline{a}) = \phi_{(r_1,r_2)}^\sharp(u_1,u_2;\underline{a}) \diamond \phi_{(r_1,r_2)}^\flat(u_1,u_2;\underline{a}).$$

\begin{example} \label{atlas_example}
Let $G = PGL_2$ and $I = \{1\}$. The atlas group $\tilde G = SL_4$ is of type $A_3$, and $\tilde I = \{1^\sharp,1^\natural,1^\flat\}$, where $1^\natural$ is the middle vertex in the Dynkin diagram of $\tilde G$. In this example, we write down explicitly the atlas maps.

An element in $\mathcal O_{(1,1)}$ can be written as $(y_{1}(a),x_1(-c))\cdot\pi(b)$. We have $u_1 = y_1(a)$ and $u_2 = x_1(-c)$. Then $\iota(\tau_{1,-}(u_2))^{-1} = \sigma_{1,+}(u_1) = e$, $\tau_{1,+}(u_2)^{-1} = x_1(c)$ and $\iota(\sigma_{1,-}(u_1)) = x_1(a)$. The atlas map is
\[
\phi_{(1,1)}(y_1(a),x_1(-c);b)
=\left\{
\begin{array}{ll}
   \tilde y_{1^\natural}(b^{-1})\tilde \alpha^\vee_{1^\natural}(b)\tilde x_{1^\flat}(c)\tilde x_{1^\sharp}(a)\dot{\tilde s}_{1^\sharp}\dot{\tilde s}_{1^\flat}\tB^+/\tB^+, &\text{if }b\neq0, \\
    \dot{\tilde s}_{1^\natural}\tilde x_{1^\flat}(c)\tilde x_{1^\sharp}(a)\dot{\tilde s}_{1^\sharp}\dot{\tilde s}_{1^\flat}\tB^+/\tB^+, &\text{if }b=0.
\end{array}
\right.
\]

Using the identity $\tilde y_{1^\natural}(b^{-1})\tilde\alpha^\vee_{1^\natural}(b) = \tilde x_{1^\natural}(b)\dot{\tilde s}_{1^\natural}\tilde x_{1^\natural}(b^{-1})$, we see that the two expressions are compatible when $b\to 0$.

An element in $\mathcal O_{(s_1,1)}$ can be written as $(x_1(a)\dot s_1,x_1(-c))\cdot\pi(b)$. We have $u_1 = x_1(a)$ and $u_2 = x_1(-c)$. Then $\iota(\tau_{1,-}(u_2))^{-1} = e$, $\sigma_{s_1,+}(u_1)\dot s_1 = x_1(a)\dot s_1$, $\tau_{1,+}(u_2)^{-1} = x_1(c)$, and $\iota(\sigma_{s_1,-}(u_1)\dot s_1) = \dot s_1^{-1}$. The atlas map is
\[
\phi_{(s_1,1)}(x_1(a),x_1(-c);b)
{}=\left\{\begin{array}{ll}
    \tilde x_{1^\flat}(a)\dot{\tilde s}_{1^\flat}\tilde y_{1^\natural}(b^{-1})\tilde\alpha^\vee_{1^\natural}(b)\tilde x_{1^\flat}(c)\dot{\tilde s}_{1^\flat}\tilde B^+/\tilde B^+ &\text{if }b\neq0, \\
    \tilde x_{1^\flat}(a)\dot{\tilde s}_{1^\flat}\dot{\tilde s}_{1^\natural}\tilde x_{1^\flat}(c)\dot{\tilde s}_{1^\flat}\tilde B^+/\tilde B^+ & \text{if }b=0.
\end{array}\right.
\]

The atlas maps over $\mathcal O_{(1,s_1)}$ and $\mathcal O_{(s_1,s_1)}$ can be computed similarly. 

\end{example}

\subsection{Images of atlas maps}
In this subsection, we  determine the images of the atlas maps and show that the atlas maps are algebraic isomorphisms onto their images. We first examine some properties of the morphisms $\sigma_{r,\pm}$ and $\tau_{r,\pm}$.

It follows from the definitions that $$\sigma_{r_1,+}(\dr_1U^-\dr_1^{-1})\dr_1\subseteq U^+\dr_1,\ \left(\tau_{r_2,+}(\dr_2U^+\dr_2^{-1})\dr_2\right)^{-1} \subseteq (\dr_2U^+)^{-1} = U^+\dr_2^{-1}.$$ Let $\tilde B_{I^\natural}^+ = \tilde L_{I^\natural}\cap\tB^+$. We have $m(\mathbb C^I)\tilde B^+_{I^\natural}/\tilde B^+_{I^\natural} =  \tilde{B}^+_{I^\natural}\dot c_{I^\natural}\tilde B^+_{I^\natural}/\tilde B^+_{I^\natural}$. Therefore,
$$\phi^\flat_{(r_1,r_2)}\left(\CO_{(r_1,r_2)}\right)\subseteq (U^+\dr_1)^\flat(\tB^+_{I^\natural}\dot c_{I^\natural}\tB_{I^\natural}^+)(U^+\dr_2^{-1}\dot w_0)^\flat\tilde B^+/\tilde B^+.$$
Similarly, we have
$$\iota\left(\sigma_{r_1,-}(\dr_1U^-\dr_1^{-1})\dr_1\right)\subseteq\iota(\dr_1U^-) = U^+\dr_1^{-1},\ \iota\left(\tau_{r_2,-}(\dr_2U^+\dr_2^{-1})\dr_2\right)^{-1} \subseteq U^+\dr_2,$$
and thus,
$$\phi^\sharp_{(r_1,r_2)}\left(\CO_{(r_1,r_2)}\right)\subseteq(U^+\dr_2)^\sharp (\tB^+_{I^\natural}\dot c_{I^\natural}\tB_{I^\natural}^+)(U^+\dr_1^{-1}\dot w_0)^\sharp\tilde B^+/\tilde B^+.$$
By construction,
\begin{align*}
\phi_{(r_1,r_2)}(\CO_{(r_1,r_2)})
&\subseteq (U^+)^\flat(U^+)^\sharp\dot r_1^\flat\dot r_2^\sharp
\tilde B^+_{I^\natural}\dot c_{I^\natural}\tilde B^+_{I^\natural}(U^+)^\flat(U^+)^\sharp
(\dot r_1^{-1}\dot w_0)^\sharp(\dot r_2^{-1}\dot w_0)^\flat
\tB^+/\tB^+\\
&\subseteq\tB^+\dot r_1^\flat\dot r_2^\sharp\dot c_{I^\natural}
(\dot r_1^{-1}\dot w_0)^\sharp(\dot r_2^{-1}\dot w_0)^\flat
\tB^+/\tB^+=\mathring{\tilde{\mathcal B}}_{\varphi(\xi_{(r_1,r_2)})}.
\end{align*}

We prove the main result of this subsection.

\begin{proposition}\label{image_atlas_map}
For $r_1,r_2\in W$, the atlas map 
\[
\phi_{(r_1,r_2)}:\CO_{(r_1,r_2)}
\to
\mathring{\tilde{\mathcal B}}_{\varphi(\xi_{(r_1,r_2)})}
\]
is an isomorphism
\begin{proof}
From the identity $\tilde y_{i^\natural}(a^{-1})\tilde\alpha_{i^\natural}(a) = \tilde x_{i^\natural}(a)\dot{\tilde s}_{i^\natural}\tilde x_{i^\natural}(a^{-1})$, we see that if we define
\[
n(\underline{a}) = \prod_{i\in I}\tilde x_{i^\natural}(a)\dot{\tilde s}_{i^\natural},
\]
then $\phi_{(r_1,r_2)}(u_1,u_2;\underline{a})$ can be written as
\[
\bigl(\iota(\tau_{r_2,-}(u_2)\dot r_2)\bigr)^{-1,\sharp}\; \bigl(\sigma_{r_1,+}(u_1)\dot r_1 \bigr)^\flat\:n(\underline{a})\; \bigl(\tau_{r_2,+}(u_2)\dot r_2\bigr)^{-1,\flat}\; \bigl(\iota(\sigma_{r_1,-}(u_1)\dot r_1 )\bigr)^\sharp\;\dot w_0^\sharp \dot w_0^\flat\tilde B^+/\tilde B^+ . \bigr.
\]
Therefore $\phi_{(r_1,r_2)}$ is algebraic.

We construct the inverse map. Since \[
\ell(r_1^\flat r_2^\sharp c_{I^\natural}(r_1^{-1}w_0)^\sharp(r_2^{-1}w_0)^\flat) = \ell(r_1)+\ell(r_2)+\ell(c_{I^\natural})+\ell(r_1^{-1}w_0)+\ell(r_2^{-1}w_0),
\]
every element in $\tilde x\tilde B^+/\tilde B^+\in \mathring{\tilde{\mathcal B}}_{\varphi(\xi_{(r_1,r_2)})}$ can be uniquely written as 
\[\tilde x\tilde B^+/\tilde B^+=
(g_1\dot r_1)^\flat (g_2\dot r_2)^\sharp n(\underline{a})(\dot r_1^{-1}h_1)^\sharp(\dot r_2^{-1}h_2)^\flat\dot w_0^\sharp\dot w_0^\flat\tilde B^+/\tilde B^+
\]
for some $g_1\in \dot r_1U^-\dot r_1^{-1}\cap U^+$, $g_2\in\dot r_2U^-\dot r_2^{-1}\cap U^+$, $\underline{a}\in\mathbb C^I$, $h_1\in \dot r_1U^+\dot r_1^{-1}\cap U^+$, and $h_2\in \dot r_2U^+\dot r_2^{-1}\cap U^+$.

The morphisms \[
(\sigma_{r_1,+},\sigma_{r_1,-}):(\dot r_1U^-\dot r_1^{-1}\cap U^+)\times (\dot r_1U^-\dot r_1^{-1}\cap U^-)\to\dot r_1U^-\dot r_1^{-1},
\]
\[
(\tau_{r_2,-},\tau_{r_2,+}):(\dot r_2U^+\dot r_2^{-1}\cap U^-)\times (\dot r_2U^+\dot r_2^{-1}\cap U^+)\to\dot r_2U^+\dot r_2^{-1}
\]
are invertible. Let \[u_1 = (\sigma_{r_1,-},\sigma_{r_1,+})^{-1}(g_1,\iota(h_1)),\enspace u_2 = (\tau_{r_2,-},\tau_{r_2,+})^{-1}(\iota(g_2)^{-1},h_2^{-1}).\]
It is easy to check that $\phi_{(r_1,r_2)}^{-1}(\tilde x\tilde B^+/\tilde B^+) = (u_1,u_2;\underline{a})$. The conclusion follows.
\end{proof}

\end{proposition}

\subsection{Compatibility with stratification}
In this subsection, we consider the compatibility of the atlas map with the stratification of $\overline G$, and conclude that $\{\phi_{(r_1,r_2)}\}_{r_1,r_2\in W}$ forms a Bruhat atlas.

\begin{proposition} \label{compatible_strata}
For every $r_1,r_2\in W$ and $\alpha\diamond\beta\in\CQ$, $\phi_{(r_1,r_2)}$ restricts to an isomorphism
$$\phi_{(r_1,r_2)}: [\alpha\diamond\beta]\bigcap\CO_{(r_1,r_2)} \xrightarrow{\sim} \mathring{\tilde{\mathcal B}}_{\varphi(\alpha\diamond\beta),\varphi(\xi_{(r_1,r_2)})}.$$
\end{proposition}

\begin{proof}
Consider $\alpha = (J;x,y)$ and $\beta = (J;v,w)$. As in \S\ref{sec:atlas_map}, we write an element in $\CO_{(r_1,r_2)}$ as $(u_1,u_2;\underline{a})$, where $u_1\in \dr_1U^-\dr_1^{-1}$, $u_2\in\dr_2U^+\dr_2^{-1}$ and $\underline{a}\in\mathbb C^I$. 

Suppose that $(u_1,u_2;\underline{a})\in [\beta]_-$. From the description of strata in \S\ref{sec:wonderful_strata}, we know that there exists $l$ in the derived subgroup of $L_J$, such that
$$u_1\dr_1t_{\underline a}\in B^-\dot v\UPJp l,\ u_2\dr_2\in B^-\dot w\UPJm l.$$
By definition, we have $\sigma_{r_1,+}(u_1)\in U^-u_1$ and $\tau_{r_2,+}(u_2)\in U^-u_2$, and thus,
$$\sigma_{r_1,+}(u_1)\dr_1t_{\underline a}\in B^-\dot v\UPJp l,\ \tau_{r_2,+}(u_2)\dr_2\in B^-\dot w\UPJm l.$$

Let $t^\natural_{\underline a} = \prod_{i\in J}\tilde\alpha_{i^\natural}^\vee(a_i)$. Then $m(\underline a) = \prod_{i\in J}\tilde y_{i^\natural}(a_i^{-1})t^\natural_{\underline a}\dot c_{(I\setminus J)^\natural}$. It follows from the construction of $\tilde G$ that
$$(\sigma_{r_1,+}(u_1)\dot r_1)^\flat \left(\prod_{i\in J}\tilde y_{i^\natural}(a_i^{-1})\right)(\sigma_{r_1,+}(u_1)\dot r_1)^{\flat,-1}\in\tilde U^-,$$
$$t_{\underline a}^{\flat,-1} t_{\underline a}^\natural (\tau_{r_2,+}(u_2)\dot r_2)^{-1,\flat}t_{\underline a}^\flat t_{\underline a}^{\natural,-1} = (\tau_{r_2,+}(u_2)\dot r_2)^{-1,\flat},$$
$$\dot c_{(I\setminus J)^\natural}l^\flat  = l^\flat\dot c_{(I\setminus J)^\natural}.$$
From these properties, we have
\begin{align*}
   \phi_{(r_1,r_2)}^\flat(u_1,u_2;\underline{a})&=\left(\sigma_{r_1,+}(u_1)\dr_1\right)^\flat m(\underline{a})\left(\tau_{r_2,+}(u_2)\dr_2\right)^{-1,\flat}\dot w_0^\flat\tilde B^+/\tilde B^+
   \\ &\in \tU^-(B^-\dot v U_{P_J^+})^\flat\dot c_{(I\setminus J)^\natural}(U_{P_J^-}\dot w^{-1})^\flat\dot w_0^\flat\tilde B^+/\tilde B^+ \\
   &\subseteq \tB^-(\dot v U_{P_J^+})^\flat\left(\dot c_{(I\setminus J)^\natural}(\UPJm)^\flat\dot c_{(I\setminus J)^\natural}^{-1}\right)\dot c_{(I\setminus J)^\natural}(\dot w^{-1}\dot w_0)^\flat \tilde B^+/\tilde B^+ \\
   & \subseteq \tB^-(\dot v U_{P_J^+})^\flat(\tilde{U}_{\tilde P_{I^\flat}^-}\cap\tilde U^-_{I^\flat\cup I^\natural})\dot c_{(I\setminus J)^\natural}(\dot w^{-1}\dot w_0)^\flat\tilde B^+/\tilde B^+ \\
   & \subseteq\tB^-(\dot v\UPJp)^\flat\dot c_{(I\setminus J)^\natural}(\dot w^{-1}\dot w_0)^\flat\tilde B^+/\tilde B^+.\\
   & \subseteq\tB^-\dot v^\flat\dot c_{(I\setminus J)^\natural}(\dot w^{-1}\dot w_0)^\flat\tB^+/\tB^+. 
\end{align*}
Then, we suppose that $(u_1,u_2;\underline{a})\in [\alpha]_+$. By a similar argument, we also have that, there exists $l'$ in the derived subgroup of $L_J$, such that
$$\iota(\tau_{r_2,-}(u_2)\dr_2)^{-1}\in B^-\dot yU_{P_J^+}\iota(l')^{-1},\ \iota(\sigma_{r_1,-}(u_1)\dr_1t_{\underline a})\in \iota(l')U_{P_J^-}\dot x^{-1}B^-,$$
and thus,
$\phi^\sharp_{(r_1,r_2)}(u_1,u_2;\underline{a})\in\tU^-\dot y^\sharp\dot c_{(I\setminus J)^\natural}(\dot x^{-1}\dot w_0)^\sharp\tB^+/\tB^+.$
By our construction of $\phi_{(r_1,r_2)}$, we have
$$\phi_{(r_1,r_2)}\left([\alpha\diamond\beta]\bigcap\CO_{(r_1,r_2)}\right)\subseteq \mathring{\tilde{\mathcal B}}_{\varphi(\alpha\diamond\beta),\varphi(\xi_{(r_1,r_2)})}.$$
Now, we have
\begin{align*}
\phi_{(r_1,r_2)}\left(\CO_{(r_1,r_2)}\right) &= \phi_{(r_1,r_2)}\left(\bigsqcup_{\alpha\diamond\beta\in\CQ}\left([\alpha\diamond\beta]\bigcap\CO_{(r_1,r_2)}\right)\right)\\
&\subseteq\bigsqcup_{\alpha\diamond\beta\in\CQ}\mathring{\tilde{\mathcal B}}_{\varphi(\alpha\diamond\beta),\varphi(\xi_{(r_1,r_2)})}\\
&=\mathring{\tilde{\CB}}_{\varphi(\xi_{(r_1,r_2)})}.
\end{align*}
We conclude from \Cref{image_atlas_map} that the inclusion above is actually equality. This completes the proof.
\end{proof}

\begin{remark}
Similar to the discussion on the torus elements in \S\ref{sec:atlas_group}, for $l\in L_J$, there is no canonical choice of elements $l^\sharp$ (resp. $l^\flat$). However, two different choices only differ by a central element in the derived subgroup of $\tilde L_{I^\sharp}$ (resp. $\tilde L_{I^\flat}$), and the argument is valid.
\end{remark}

Together with \Cref{atoms}, we obtain

\begin{corollary} \label{nonempty_intersection}
For every $r_1,r_2\in W$ and $\alpha\diamond\beta = (J;x,y;v,w)\in\CQ$, we have $[\alpha\diamond\beta]\cap\CO_{(r_1,r_2)}\neq\emptyset$ if and only if there exist $z_1,z_2\in W_J$, such that
$vz_2\leq r_1\leq xz_1$, and $yz_1\leq r_2\leq wz_2$.
\end{corollary}

\section{Compatibility with total positivity}\label{section:positivity}

In this section, we prove that the Bruhat atlas constructed in
\S\ref{section:bruhat_atlas} is compatible with total positivity.
We use the cell decomposition in
\cite{He04} to show that every chart meeting a stratum contains
its nonnegative part, and then we identify the image of the top positive cell. A general positive-atlas argument gives the closure relations of totally nonnegative cells and the manifold structure for the closure of every cell.

\subsection{Remarkable polyhedral spaces and regularity property}
Lusztig described the totally nonnegative flag variety as a
remarkable polyhedral space. Based on later works, we provide the following definition. 

\begin{definition}
Let $X=\bigsqcup_{\gamma}X_\gamma$ be a stratification of a real
algebraic variety, or of an ind-variety with finite-dimensional strata. Let $X_{\ge0}\subseteq X(\mathbb R)$ be a
closed connected subset, and set
$X_{\gamma,>0}=X_{\ge0}\cap X_\gamma(\mathbb R)$.
We call $X_{\ge0}=\bigsqcup_\gamma X_{\gamma,>0}$ a
\emph{remarkable polyhedral space} if the following conditions
hold:
\begin{enumerate}
    \item (Connected components) $X_{\gamma,>0}$ is either empty or a connected component of $X_\gamma(\mathbb R)$;
    \item (Cell structure) If $X_{\gamma,>0} \neq \emptyset$, then $X_{\gamma,>0} \simeq\mathbb R_{>0}^{d_\gamma}$ is a semi-algebraic cell, where $d_\gamma=\dim_{\mathbb R} X_\gamma$;
    \item (Cellular decomposition) The Hausdorff closure $\overline{X_{\gamma,>0}}$ is the disjoint union of totally nonnegative cells $\bigsqcup_{\gamma'\leq\gamma} X_{\gamma',>0}$, where $\gamma'\leq\gamma$ if $X_{\gamma'}\subseteq\overline{X_{\gamma}}$.
\end{enumerate}
We say that $X_{\ge0}$ has the \emph{regularity property} if,
for every nonempty $X_{\gamma,>0}$, there is a homeomorphism from a closed ball to $\overline{X_{\gamma,>0}}$ that takes the interior onto $X_{\gamma,>0}$.

\end{definition}

Examples of remarkable polyhedral spaces with the regularity
property include:
\begin{itemize}
    \item totally nonnegative (partial) flag varieties \cite{GKL22, BH24};
    \item totally nonnegative twisted flag varieties \cite{HX26};
    \item totally nonnegative double flag varieties \cite{BH22, HX26};
    \item links of the identity in totally nonnegative reduced double Bruhat cells \cite{Her14, BH22,HX26};
    \item totally nonnegative matroid Schubert varieties \cite{HSX25}.
\end{itemize}

We will use the regularity theorem for totally nonnegative
Kac--Moody flag varieties in \cite{BH24}. The \emph{totally nonnegative flag variety} is defined to be
\[
\tilde{\mathcal B}_{\ge0}
 :=\overline{\tilde G_{\ge0}\tilde B^+/\tilde B^+}
  =\overline{\tilde U^-_{\ge0}\tilde B^+/\tilde B^+}\text{  (Hausdorff closure).}
\]
For $\tilde v\leq\tilde w$, set
$\tilde{\mathcal B}_{\tilde v,\tilde w,>0}
 :=\tilde{\mathcal B}_{\ge0}
   \cap\mathring{\tilde{\mathcal B}}_{\tilde v,\tilde w}(\mathbb R).$

\begin{theorem}[{\cite[Theorem B \& Theorem D]{BH24}}]
The totally nonnegative flag variety
\[
\tilde{\mathcal B}_{\ge0}
 =\bigsqcup_{\tilde v\leq\tilde w}
   \tilde{\mathcal B}_{\tilde v,\tilde w,>0}
\]
is a remarkable polyhedral space with regularity property.
\end{theorem}

\subsection{Positive atlases}\label{sec:positive_atlas}

Let $M$ be a finite-dimensional irreducible real algebraic
variety with a finite stratification
$M=\bigsqcup_{y\in\mathcal Y}\mathring M_y$. Let $M_{\ge0}\subseteq M(\mathbb R)$ be a connected closed subset of $M(\mathbb R)$. For $y\in\CY$, set $M_{y,>0}:= M_{\ge0}\cap\mathring M_y$. Let $y_0\in\CY$ such that $\mathring M_{y_0}$ is the open stratum. We assume that $M_{\ge0} = \overline{M_{y_0,>0}}$.

We call the Bruhat atlas a \emph{positive atlas} on
$(M,\mathcal Y,M_{\ge0})$ if
\begin{itemize}
\item $M_{y,>0}\subseteq U_f(\mathbb R)$ whenever
$U_f\cap\mathring M_y\neq\emptyset$;
\item for every $f$, the image $\phi_f(M_{y_0,>0})$ is the
totally nonnegative cell in the open Richardson variety
$\phi_f(U_f\cap\mathring M_{y_0})$.
\end{itemize}
A positive atlas transfers the
local topological properties from
$\tilde{\mathcal B}_{\ge0}$ to $M_{\ge0}$, as shown in the following proposition.

\begin{proposition}\label{positive_atlas}
Suppose that $\{\phi_f\}$ is a positive atlas on
$(M,\mathcal Y,M_{\ge0})$. Then every $M_{y,>0}$ is nonempty,
and
$M_{\ge0}=\bigsqcup_{y\in\mathcal Y}M_{y,>0}$ is a
remarkable polyhedral space. Moreover, for every $y$, the
Hausdorff closure $\overline{M_{y,>0}}$ is a topological
manifold with boundary
\[
\partial\overline{M_{y,>0}}
   =\bigsqcup_{y'<y}M_{y',>0}.
\]
\end{proposition}

\begin{proof}
We denote $\bar\phi_f(y)$ a Bruhat interval of $\tilde W$ such that $\phi_f(\mathring M_y\cap U_f) =\mathring{\tilde\CB}_{\bar\phi_f(y)}$. For $y\in\CY$, we take $f\in F$ such that $\mathring M_y\cap U_f\neq\emptyset$. By assumption, we have $\phi_f(M_{y_0,>0}) =\tilde\CB_{\bar\phi_f(y_0),>0}$. Since $\phi_f$ is a stratified embedding, we have
$$\phi_f(M_{y,>0})=\phi_f\left(\overline{M_{y_0,>0}}\bigcap\mathring M_y\right) = \overline{\phi_f(M_{y_0,>0})}\bigcap \mathring{\tilde\CB}_{\bar\phi_f(y)} = \tilde\CB_{\bar\phi_f(y),>0}.$$
Therefore, $M_{y,>0}$ is a connected component of $\mathring M_y\cap U_f$ and a semi-algebraic cell. We also have $$\dim M_{y,>0} = \dim \tilde\CB_{\bar\phi_f(y)}=\dim(\mathring M_y\cap U_f).$$
Since $\{U_f\}_{f\in F}$ is an open covering of $M$, we have $\dim\mathring M_y(\mathbb R) = \dim M_{y,>0}$.

Since $U_f$ is open in $M$, we also have $M_{y,>0}$ is open in $\mathring M_y(\mathbb R)$. Since $M_{y,>0} = \mathring M_y\cap M_{\ge0}$ is closed, $M_{y,>0}$ is a connected component of $\mathring M_y(\mathbb R)$.

For $y,y'\in\CY$ such that $y'\leq y$, since $\{U_f\}_{f\in F}$ is an open covering of $M$, there is an $f\in F$ such that $U_f\cap\mathring M_{y'}\neq\emptyset$. Since $\mathring M_{y'}$ is contained in the Zariski closure of $\mathring M_y$, we have $U_f\cap\mathring M_y\neq\emptyset$. 
Since $\tilde\CB_{\ge0}$ is a remarkable polyhedral space, we have 
$$\phi_f(M_{y',>0})\subseteq \phi_f\left(\overline{M_{y,>0}}\bigcap U_f\right),$$
which shows that $M_{y',>0}\subseteq\overline{M_{y,>0}}$. Since $\overline{\mathring{M}}_y = \bigsqcup_{y'\leq y}\mathring{M}_{y'}$, we have $\overline{M_{y,>0}}\subseteq\bigsqcup_{y'\leq y}M_{y',>0}$. This verifies the closure relation condition.

Finally, since $\overline{\phi_f(M_{y,>0})}$ is a closed ball with interior $\phi_f(M_{y,>0})$, the collection $\left(\overline{M_{y,>0}}\cap U_f,\phi_f\right)_{f\in F}$ gives an Euclidean chart on $\overline{M_{y,>0}}$, making it a topological manifold, with $M_{y,>0}$ lying in the interior. For $y'<y$, we have 
$$\phi_f(M_{y',>0})\subseteq\partial\overline{\phi_f(M_{y,>0})} = \partial\left(\phi_f\left(\overline{M_{y,>0}}\bigcap U_f\right)\right). $$
Therefore, $\partial\overline{M_{y,>0}} = \bigsqcup_{y'<y} M_{y',>0}$.
\end{proof}

\begin{remark}
For simplicity, we only consider $M$ a finite-dimensional varieties. Similar arguments apply to
ind-varieties with suitable extra conditions. For
example, Bao and He \cite{BH21b,BH24} constructed positive
atlases for Kac--Moody partial flag varieties with target space being certain $J$-totally nonnegative flag varieties. We do not
pursue this generality here.
\end{remark}

Now we state the main result of this subsection.

\begin{theorem}\label{positive_atlas_for_compactfication}
The maps $\{\phi_{(r_1,r_2)}\}_{r_1,r_2\in W}$ form a positive atlas on $(\overline G,\overline{G_{\ge0}})$.
\end{theorem}

\subsection{Marsh--Rietsch parametrizations}
\label{sec:cell_parametrizations}

Let $v,w\in W$ such that $v\leq w$. Fix a reduced expression $\mathbf w=(s_{i_1},\ldots,s_{i_n})$ of $w$. A subexpression is a sequence
$\mathbf v=(t_1,\ldots,t_n)$ with
$t_k\in\{1,s_{i_k}\}$.
It is called positive if
\[
t_1\cdots t_{k-1}
   <t_1\cdots t_{k-1}s_{i_k},
   \enspace 1\leq k\leq n.
\]
By \cite[Lemma~3.5]{MR04}, there is a unique positive subexpression with product $v$,
which we denote by $\mathbf v_+$.

Define
\[
G_{\mathbf v_+,\mathbf w,>0}
=
\left\{
g_1\cdots g_n \bigg|
\begin{array}{ll}
g_k=\dot s_{i_k},&\text{if }t_k=s_{i_k},\\
g_k=y_{i_k}(a_k),\ a_k>0,&\text{if }t_k=1
\end{array}
\right\}.
\]
It is obvious that this set is isomorphic to
$\mathbb R_{>0}^{\ell(w)-\ell(v)}$. It was proved in \cite[Theorem~11.3]{MR04} for groups of finite types and in \cite[Theorem~4.10]{BH21a} for Kac--Moody groups that 
\[
\mathcal B_{v,w,>0} = G_{\mathbf v_+,\mathbf w,>0}B^+/B^+.
\]

Note that $G_{\mathbf 1_+,\mathbf w,>0} = U_{w,>0}^-$. We will also use the notations
\[
\begin{aligned}
U^-_{w,<0}
&=\{y_{i_1}(a_1)\cdots y_{i_n}(a_n)|
       a_1,\ldots,a_n<0\} = \left(U_{w^{-1},>0}^-\right)^{-1},\\
U^+_{w,<0}
&=\{x_{i_1}(a_1)\cdots x_{i_n}(a_n)|
       a_1,\ldots,a_n<0\} = \left(U_{w^{-1},>0}^+\right)^{-1}.
\end{aligned}
\]

\subsection{Some operations in Weyl groups} \label{sec:Weyl_operations}
For $w,v\in W$, define
\begin{gather*}
    w*v = \max\{w'v'|w'\leq w,v'\leq v\},
\\  w\circ_l v = \min\{w'v|w'\leq w\}, \quad w\circ_r v = \min\{wv'|v'\leq v\}.
\end{gather*}
The existence of these elements follows from \cite[Lemma 1.4]{He07}. We recall some basic properties. For a proof, we refer to \cite[Appendix A]{HL11}.
\begin{lemma} \label{operations}
For $w,v,u\in W$, we have
\begin{enumerate}
    \item $(w*v)\circ_l u = w\circ_l(v\circ_l u),\enspace w\circ_r(v*u) = (w\circ_r v)\circ_r u$;
    \item $w\circ_l v\leq u$ if and only if $v\leq w^{-1}*u$;
    \item $w\circ_r v\leq u$ if and only if $w\leq u*v^{-1}$.
\end{enumerate}
\end{lemma}

\begin{remark}
Lusztig \cite[Corollary 2.8]{L94} showed that for $w_1,w_2\in W$, we have $U_{w_1,>0}^\pm U_{w_2,>0}^\pm= U_{w_1*w_2,>0}^\pm$.
\end{remark}

We have two lemmas describing some Zariski open subsets that a totally nonnegative cell lies in. The first one follows from \cite[\S3.3(a) \& Lemma 5.1]{BH24}. 

\begin{lemma} \label{inside_open_flag}
Suppose that $v,r,w\in W$ with $v\leq w$. Then we have
$$G_{\mathbf v_+,\mathbf w,>0}B^+/B^+\subseteq \dot rB^-B^+/B^+$$
if and only if $v\leq r\leq w$.

\end{lemma}

\begin{lemma} \label{inside_open_group}
Suppose that $v,w,r_1,r_2\in W$. Then we have
$$U^-_{w,>0}T_{>0}U^+_{v,>0}\subseteq \dot r_1B^-B^+\dot r_2^{-1}$$
if and only if $w\geq r_1\circ_r r_2^{-1}$ and $v\geq r_1\circ_l r_2^{-1}$.

\begin{proof}
Let $w' = w* r_2$, $v' = v\circ_l r_2$. It follows from \cite[Proposition 5.2]{BH21a} that
$$U^-_{w,>0}T_{>0}U_{v,>0}^+\dot r_2B^+/B^+\subseteq G_{\mathbf{v'}_+,\mathbf{w'},>0}B^+/B^+.$$ 
Suppose that $v'\leq r_1\leq w'$. By \Cref{inside_open_flag} we have
$$U^-_{w,>0}T_{>0}U^+_{v,>0}\dot r_2B^+/B^+\subseteq\dot r_1B^-B^+/B^+.$$
On the other hand, if $U^-_{w,>0}T_{>0}U^+_{v,>0}\dot r_2B^+/B^+\subseteq\dot r_1B^-B^+/B^+$, then $$\mathring{\mathcal B}_{v',w'}\bigcap\dot r_1B^-B^+/B^+ \neq \emptyset,$$ 
and by \cite[\S3.3 (a)]{BH24}, we know that $v'\leq r_1\leq w'$.

Finally, by \Cref{operations}, $v'\leq r_1\leq w'$ is equivalent to $w\geq r_1\circ_r r_2^{-1}$ and $v\geq r_1\circ_l r_2^{-1}$.
\end{proof}

\end{lemma}

\subsection{Levi factors of some parametrizing sets} \label{sec:levi_residue}
In this subsection, we assume that $G$ is of finite type. Let $J\subseteq I$. Let $p_J^\pm: P_J^\pm\to L_J$ be the natural projection from the parabolic subgroups to their Levi subgroup. For an element $g = b_-b_+$ where $b_\pm\in B^\pm$, we denote $\pi_J(g):=p_J^-(b_-)p_J^+(b_+)$. It is clear that $\pi_J$ is well-defined. Note that $\pi_J(g)\in B_J^-B_J^+$ where $B_J^\pm = B^\pm\cap L_J$. 

For $v,w\in W^J$ and $v\leq w$, we define the \emph{$J$-residue} of $v$ and $w$ to be
$$\text{Res}_J(v,w) := \max\{z\in W_J|vz\leq w\}.$$
The existence of $\text{Res}_J(v,w)$ follows from \cite[Lemma 11(ii)]{LLM02}.

\begin{lemma} \label{levi_nonnegative}
Suppose that $v,r,w\in W^J$ such that $v\leq r\leq w$. 
Let $a=\text{Res}_J(v,r)$ and $b=\text{Res}_J(r,w)$.
For $g\in G_{\mathbf v_+,\mathbf w,>0}$, we have $\pi_J(\dot r^{-1}g)\in U_{b,>0}^-T_{>0}U_{a^{-1},>0}^+$.

\begin{proof}
We prove the statement by induction on $l(w)$. The base case where $l(w) = 0$ is trivial. 

Let $\mathbf w = (s_{i_1},s_{i_2},\cdots, s_{i_n})$, and write $g = g_1g_2\cdots g_n$ as in the definition of $G_{\mathbf v_+,\mathbf w,>0}$. Set $w' = s_{i_1}w$, $r' = s_{i_1}\circ_l r$, $v' = s_{i_1}\circ_l v$, $a' = \text{Res}_J(v',r')$, and $b' = \text{Res}_J(r',w')$. Then $v',r',w'\in W^J$ and $v'\leq r'\leq w'$.

We divide the computation into several cases.
\begin{enumerate}[label = (\roman*)]
    \item $s_{i_1}r>r$ and $s_{i_1}r\in W^J$. Then $r' = r$ and $v\leq r\leq s_{i_1}w$. This forces $g_1\in y_{i_1}(\mathbb R_{>0})$ and $b' = b$. In this case, we have $\dot r^{-1}y_{i_1}(\mathbb R_{>0})\dot r\subseteq U_{P_J^-}$, and 
    $$\dot r^{-1}G_{\mathbf v_+,\mathbf w,>0}\subseteq U_{P_J^-}\left(\dot r^{\prime,-1}G_{\mathbf v_+,\mathbf w',>0}\right).$$

    \item $s_{i_1}r = rs_j$ for some $j\in J$. We still have $r' = r$, $v\leq r\leq s_{i_1}w$, and $g_1\in y_{i_1}(\mathbb R_{>0})$. Now $b = s_j*b'$. In this case, we have $\dot r^{-1}y_{i_1}(\mathbb R_{>0})\dot r = y_j(\mathbb R_{>0})$. Here, the sign is preserved according to \cite[Lemma 1.3.12]{He05}. Then
    $$\dot r^{-1}G_{\mathbf v_+,\mathbf w,>0}\subseteq U^-_{s_j,>0}\left(\dot r^{\prime,-1}G_{\mathbf v_+,\mathbf w',>0}\right).$$

    \item $s_{i_1}r<r$ and $g_1 = \dot s_{i_1}$. Then $r' = s_{i_1}r\in W^J$, $v' = s_{i_1}v$, and $a' = a$, $b' = b$. We have 
        $$\dot r^{-1}G_{\mathbf v_+,\mathbf w,>0}=\dot r^{\prime,-1}G_{\mathbf v'_+,\mathbf w',>0}.$$

    \item $s_{i_1}r<r$, $g_1= y_{i_1}(m)$ for $m>0$, and $s_{i_1}v\in W^J$. In this case, $r'=s_{i_1}r\in W^J$, $a' = a$, and $b' = b$.
    
    Using the identity $$y_{i_1}(m) = x_{i_1}(m^{-1})\dot s_{i_1}x_{i_1}(m)\alpha^\vee_{i_1}(m),$$
    we have $$\dot r^{-1}gB^+/B^+ = (\dot r^{-1}x_{i_1}(m^{-1})\dot r)\dot r^{\prime,-1}x_{i_1}(m)g'B^+/B^+,$$
    for some $g'\in G_{\mathbf v_+,\mathbf w',>0}$. By assumption, $\dot r^{-1}x_{i_1}(m^{-1})\dot r\in U_{P_J^-}$.
    
    Let $w_{J,0}$ denote the longest element in $W_J$. Since $v,v',w'\in W^J$, we have $$G_{\mathbf{vw_{J,0}}_+,\mathbf{w'w_{J,0}},>0} = G_{\mathbf v_+,\mathbf{w'},>0}\dot w_{J,0},\enspace G_{\mathbf{v'w_{J,0}}_+,\mathbf{w'w_{J,0}},>0} = G_{\mathbf v'_+,\mathbf{w'},>0}\dot w_{J,0}.$$ By \cite[Proposition 5.2]{BH21a}, we have
    \begin{equation*}
        x_{i_1}(m)g'\dot w_{J,0}B^+/B^+ = h\dot w_{J,0}B^+/B^+,\text{ for some }h\in G_{\mathbf{v'}_+,\mathbf{w'},>0}.
    \end{equation*}
    Since $w'\in W^J$, we may consider the map 
    \[
    \begin{aligned}
    \pi^{w'w_{J,0}}_{w'}:\quad
    U^+\dot w'\dot w_{J,0}B^+/B^+
       &\longrightarrow U^+\dot w'B^+/B^+,\\
    u\dot w'\dot w_{J,0}B^+/B^+
       &\longmapsto u\dot w'B^+/B^+.
    \end{aligned}
    \]
    as in \cite[\S4.3]{MR04}. Since $x_{i_1}(m)\in B^+$, we have
    \begin{align*}
      x_{i_1}(m)g'B^+/B^+ &= \pi^{w'w_{J,0}}_{w'}(x_{i_1}(m)g'\dot w_{J,0}B^+/B^+) \nonumber \\
      &=\pi^{w'w_{J,0}}_{w'}(h\dot w_{J,0}B^+/B^+)\nonumber \\
      &= hB^+/B^+.\
    \end{align*}
    It follows that we have 
    $$h^{-1}x_{i_1}(m)g'\in B^+\cap \dot w_{J,0}B^+\dot w_{J,0}^{-1} = TU_{P_J^+}.$$
    Write $x_{i_1}(m)g' = htu$, where $t\in T$ and $u\in U_{P_J^+}$. We show that $t\in T_{>0}$.

    For this purpose, we may assume that $G$ is simply-laced, because for groups of other types the result can be obtained from folding. Let $\lambda$ be a regular dominant weight, $V_\lambda$ the corresponding irreducible highest weight module and $\eta_\lambda$ a highest weight vector. Then we have
    $$x_{i_1}(m)g'\eta_\lambda =\lambda(t) h\eta_\lambda.$$
    Since $x_{i_1}(m)g'B^+\in\mathcal B_{\ge0}$ and $hB^+\in\mathcal B_{\ge0}$, the vectors $x_{i_1}(m)g'\eta_\lambda$ and $h\eta_\lambda$ are both nonnegative linear combinations of canonical basis \cite[Proposition 8.17 (b)]{L94}, and we have $\lambda(t)>0$. Since $\lambda$ is arbitrary, we have $t\in T_{>0}$.

    In conclusion, we have
    $$\dot r^{-1}G_{\mathbf v_+,\mathbf w,>0}\subseteq U_{P_J^-}\left(\dot r^{\prime,-1}G_{\mathbf {v'}_+,\mathbf{w'},>0}\right)T_{>0}U_{P_J^+}.$$

    \item $s_{i_1}r<r$, $g_1\in y_{i_1}(\mathbb R_{>0})$, and $s_{i_1}v=vs_j$ for $j\in J$. In this case, $r' = s_{i_1}r\in W^J$, $v' = v$, $a = s_j*a'$, and $b' = b$. 

    Similar to (iv), we have
    $$\dot r^{-1}G_{\mathbf v_+,\mathbf w,>0}\subseteq U_{P_J^-}\dot r^{\prime,-1}x_{i_1}(\mathbb R_{>0})G_{\mathbf v_+,\mathbf{w'},>0}.$$
    It follows from \cite[Lemma 1.3.12 \& Lemma 2.2.4]{He05} that
    $$x_{i_1}(\mathbb R_{>0})G_{\mathbf v_+,\mathbf{w'},>0}\subseteq G_{\mathbf v_+,\mathbf w',>0}x_j(\mathbb R_{>0})T_{>0}U_{P_J^+}.$$
    See also \cite[Lemma 4.6]{HX26} for a similar argument. Therefore, we have
    $$\dot r^{-1}G_{\mathbf v_+,\mathbf w,>0}\subseteq U_{P_J^-}\left(\dot r^{\prime,-1}G_{\mathbf{v}_+,\mathbf{w'},>0}\right)U^+_{s_j,>0}T_{>0}U_{P_J^+}.$$
\end{enumerate}
The statement then follows from the inductive hypothesis on $w'$.
\end{proof}

\end{lemma}

\subsection{Totally nonnegative cells in charts}
\label{sec:nonnegative_cells}

We now show that the nonnegative part of a stratum lies in every chart meeting that stratum. Here we use the factorization of elements in $\overline{G_{\ge0}}$ from \cite{He04}.

\begin{lemma}\label{cell}
Let $(J;x,y;v,w)\in\CQ$.
Every point of $[J;x,y;v,w]_{>0}$ can be written as $(g_1,\iota(g_2)^{-1}g_3^{-1})\cdot z_J$,
where $g_1\in G_{\mathbf v_+,\mathbf x,>0}$, $g_2\in G_{\mathbf{y^J}_+,\mathbf{w^J},>0}$, $g_3\in U^-_{y_J^{-1},>0}T_{>0}U^+_{w_J^{-1},>0}$.
\end{lemma}

\begin{proof}
By \cite[Theorem~3.14 \& Corollary~3.16]{He04}, every point
of $\overline{G_{\ge0}}$ has such a factorization for some
$J$, some pairs $v\leq x$ and $y,w\in W$ such that $y^J\leq w^J$. In particular, this is a cell decomposition. We show that each cell lies in a unique stratum.

The Marsh-Rietsch parametrization satisfies
\[
G_{\mathbf v_+,\mathbf x,>0}
   \subseteq B^+\dot xB^+\cap U^-\dot v,\enspace 
\iota\bigl(
G_{\mathbf{y^J}_+,\mathbf{w^J},>0}
\bigr)^{-1}
   \subseteq U^+\dot{y}^J
      \cap B^-\dot{w}^JB^-.
\]
We also have
\[
U^-_{y_J^{-1},>0}T_{>0}U^+_{w_J^{-1},>0}
\subseteq
B_J^+\dot{y}_J^{-1}B_J^+
\cap
B_J^-\dot{w}_J^{-1}B_J^-.
\]
Using the stabilizer of $z_J$ and the assumptions
$x,v\in W^J$, these containments show that the resulting
point lies in
\[
(B^+,B^+)(\dot x,\dot y)\cdot z_J
\ \bigcap\
(B^-,B^-)(\dot v,\dot w)\cdot z_J
=[J;x,y;v,w].
\]
The conclusion follows since the strata are disjoint.
\end{proof}

We apply this parameterization together with the lemmas in
\S\ref{sec:Weyl_operations} and \S\ref{sec:levi_residue}
to prove the main result of this subsection.

\begin{proposition} \label{inside_open_compactification}
Suppose that $\CO_{(r_1,r_2)}\cap[\alpha\diamond\beta]\neq\emptyset$. Then $[\alpha\diamond\beta]_{>0}\subseteq\CO_{(r_1,r_2)}$.

\begin{proof}
Denote $\alpha\diamond\beta = (J;x,y;v,w)$. Recall from \Cref{nonempty_intersection} that when $\CO_{(r_1,r_2)}\cap [\alpha\diamond\beta]\neq\emptyset$, we have $z_1,z_2\in W_J$, such that
$vz_2\leq r_1\leq xz_1$, and $yz_1\leq r_2\leq wz_2$. From this we have $v\leq r_1^J\leq x$ and $y^J\leq r_2^J\leq w^J$.

We first deduce some conditions on the parabolic parts. Recall from \S\ref{sec:levi_residue} the definition of $J$-residue: for $v,w\in W^J$ with $v\leq w$, we define
$$\text{Res}_J(v,w) := \max\{z\in W_J|vz\leq w\}.$$
Let $a_1 = \text{Res}_J(v,r_1^J)$, $b_1 = \text{Res}_J(r_1^J,x)$, $a_2 = \text{Res}_J(y^J,r_2^J)$, $b_2 = \text{Res}_J(r_2^J,w^J)$.
Then the condition
$vz_2\leq r_1\leq xz_1,\enspace yz_1\leq r_2\leq wz_2$
implies that 
\[z_2\leq a_1*r_{1,J},\enspace r_{1,J}\leq b_1*z_1,\]
\[y_J\circ_l z_1\leq y_Jz_1\leq a_2*r_{2,J},\enspace r_{2,J}\leq b_2*(w_Jz_2)\leq b_2*w_J*z_2.\]
From \Cref{operations}, we have $z_1\leq y_J^{-1}*a_2*r_{2,J}$. Combining these inequalities, we obtain
$r_{2,J}\leq b_2*w_J*a_1*r_{1,J},\enspace r_{1,J}\leq b_1*y_J^{-1}*a_2*r_{2,J},$
or equivalently,
\begin{equation}
r_{1,J}\circ_l r_{2,J}^{-1}\leq a_1^{-1}*w_J^{-1}*b_2^{-1},\enspace r_{1,J}\circ_r r_{2,J}^{-1}\leq b_1*y_J^{-1}*a_2.    \tag{3.1}
\end{equation}

Now, we turn to compute $(\dot r_1^{-1},\dot r_2^{-1})\cdot [\alpha\diamond\beta]_{>0}$. By \Cref{inside_open_flag} and \Cref{levi_nonnegative}, we have
\begin{gather*}\dot r_1^{J,-1}G_{\mathbf v_+,\mathbf x,>0}\subseteq U_{P_J^-}\left(U^-_{b_1,>0}T_{>0}U^+_{a_1^{-1},>0}\right)U_{P_J^+},\\ \dot r_2^{J,-1}G_{\mathbf {y^J}_+,\mathbf {w^J},>0}\subseteq U_{P_J^-}\left(U^-_{b_2,>0}T_{>0}U^+_{a_2^{-1},>0}\right)U_{P_J^+}.\end{gather*}
By \Cref{cell}, an arbitrary point of
$[\alpha\diamond\beta]_{>0}$ has the form $(g_1,\iota(g_2)^{-1}g_3^{-1})\cdot z_J,$
with
$g_1\in G_{\mathbf v_+,\mathbf x,>0},
g_2\in G_{\mathbf{y^J}_+,\mathbf{w^J},>0},
g_3\in U^-_{y_J^{-1},>0}T_{>0}U^+_{w_J^{-1},>0}.$

For such a point, we have
\begin{align*}
&(\dot r_1^{-1}g_1,\dot r_2^{-1}\iota(g_2)^{-1}g_3^{-1})\cdot z_J=(\dot r_1^{-1}g_1g_3,\iota(\dot r_2^{-1}g_2)^{-1})\cdot z_J \\
\in &\left(U_{P_J^-}\dot r_{1,J}^{-1}U^-_{b_1*y_J^{-1},>0}T_{>0}U^+_{a_1^{-1}*w_J^{-1},>0}, U_{P_J^+}\dot r_{2,J}^{-1}U^+_{b_2,<0}T_{>0}U^-_{a_2^{-1},<0}\right)\cdot z_J.
\end{align*}
To show that $(\dot r_1^{-1},\dot r_2^{-1})\cdot [\alpha\diamond\beta]_{>0}\subseteq\mathcal O_{(1,1)}$, or equivalently, $[\alpha\diamond\beta]_{>0}\subseteq\mathcal O_{(r_1,r_2)}$, it remains to show that
$$\dot r_{1,J}^{-1}U^-_{b_1*y_J^{-1}*a_2,>0}T_{>0}U_{a_1^{-1}*w_J^{-1}*b_2^{-1},>0}^+\dot r_{2,J}\subseteq B_J^-B_J^+,$$
and this follows directly from (3.1) and \Cref{inside_open_group}.
\end{proof}

\end{proposition}

\subsection{Images of nonnegative cells}
In this subsection, we consider the atlas map $\phi_{(r_1,r_2)}$ restricted to $G$, and compute the image of some totally nonnegative cells in $G_{\ge0}$. As a consequence, we show that the collection of maps $\{\phi_{(r_1,r_2)}\}$ is a positive atlas on $(\overline{G},\overline{G_{\ge0}})$.

Recall that $\text{Stab}_{G\times G}(z_I) = G_{\text{diag}}$. We may identify $(G\times G)\cdot z_I$ with $G$ by $(g_1,g_2)\to g_1g_2^{-1}$. Under this identification, we have
\[
\CO_{(r_1,r_2)}\bigcap (G\times G)\cdot z_I
\xrightarrow{\sim}\dot r_1U^-TU^+\dot r_2^{-1},
\]
\[
[I;1,w_1^{-1};1,w_2^{-1}]
\xrightarrow{\sim}B^+\dot w_1B^+\cap B^-\dot w_2B^-.
\]
In this subsection, we consider the restriction of $\phi_{(r_1,r_2)}$ to $\dot r_1U^-TU^+\dot r_2^{-1}$.

\begin{proposition} \label{image_positive}
Suppose $w_1,w_2\in W$ such that $r_1\leq w_1$, $\ell(r^{-1}_1w_1) = \ell(w_1)-\ell(r_1)$, and $r_2\leq w_2$, $\ell(w_2r_2) = \ell(w_2)-\ell(r_2)$. Then
\[
\phi_{(r_1,r_2)}(U^-_{w_1,>0}T_{>0}U^+_{w_2,>0})
=\tilde{\CB}_{\varphi((I;1,w_1^{-1})\diamond(I;1,w_2^{-1})),\varphi((\emptyset;r_1,r_2)\diamond(\emptyset;r_1,r_2)),>0}.
\]
In particular, $\phi_{(r_1,r_2)}(U_{w_0,>0}^-T_{>0}U^+_{w_0,>0})$ is a totally nonnegative cell in $\tilde{\mathcal B}_{\ge0}$.

\begin{proof}
We first write an element in $U^-_{w_1,>0}T_{>0}U^+_{w_2,>0}$ in the form of $u_1\dr_1 t\dr_2^{-1}u_2^{-1}$ for $u_1\in\dr_1U^-\dr_1^{-1}$, $t\in T$ and $u_2\in\dr_2U^+\dr_2^{-1}$. For this, we consider $h_1\in U^-_{w_1,>0}$, $h_2\in U^+_{w_2^{-1},<0}$ and $t'\in T_{>0}$. 

By \cite[Lemma 5.6]{GKL22} (see also \cite[\S6.3]{BH24}), there exist $u_-\in U^-\cap\dr_1^{-1}U^+\dr_1$, $g_1\in U_{r_1^{-1}w_1,>0}^-$, $k_1\in U_{r_1^{-1},>0}^+$ and $t_1\in T_{>0}$ such that $\dr_1^{-1}h_1 = u_-g_1k_1t_1$. Similarly, there exist $u_+\in U^+\cap\dr_2^{-1}U^-\dr_2$, $g_2\in U_{r_2^{-1}w_2^{-1},<0}^+$, $k_2\in U_{r_2^{-1},<0}^-$ and $t_2\in T_{>0}$ such that $\dr_2^{-1}h_2 = u_+g_2k_2t_2$.

An element in $U^-_{w_1,>0}T_{>0}U^+_{w_2,>0}$ can be written as
$$h_1t'h_2^{-1} = \dr_1u_-g_1k_1t_1t't_2^{-1}k_2^{-1}g_2^{-1}u_+^{-1}\dr_2^{-1}.$$
Let $k_1'\in U^+_{r_1^{-1},>0}$ and $k_2'\in U^-_{r_2^{-1},<0}$ and $t\in T_{>0}$ such that
$$k_1t_1t't_2^{-1}k_2^{-1} = (k_2')^{-1}tk_1'.$$
Then we may write
$$h_1t'h_2^{-1} = \dr_1u_-g_1(k_2')^{-1}tk_1'g_2^{-1}u_+^{-1}\dr_2^{-1}.$$

Take $u_1 = \dr_1(u_-g_1(k_2')^{-1})\dr_1^{-1}$ and $u_2 = \dr_2(u_+g_2(k_1')^{-1})\dr_2^{-1}$. We first compute $\sigma_{r_1,+}(u_1)\dr_1t$ and $\tau_{r_2,+}(u_2)\dr_2$. 

Note that, by definition, 
$$u_1\dr_1B^+ = h_1t'h_2^{-1}\dot r_2B^+ \in  U^-_{w_1,>0}T_{>0}U^+_{w_2,>0}\dot r_2B^+.$$
Furthermore, by \cite[Proposition 5.2]{BH21a}, we have
$$u_1\dot r_1B^+/B^+\in U^-_{w_1,>0}T_{>0}U^+_{w_2,>0}\dot r_2B^+\subseteq U^-_{w_1*r_2,>0}B^+/B^+.$$
By \cite[Theorem 5.3 (1)]{BH24}, $\sigma_{r_1,+}(u_1)\dr_1\in h_1'B^+$ for some $h_1'\in U_{r_1,>0}^-$. We also have 
\begin{align*}
 \sigma_{r_1,+}(u_1)\dr_1\in U^-u_1\dr_1 &= U^-\dr_1u_-g_1(k_2')^{-1}
 \\
 &= U^-\dr_1u_-g_1k_1t_1(t_1^{-1}k_1^{-1}(k'_2)^{-1}) \\
 &= U^-h_1(t_1^{-1}k_1^{-1}(k_2')^{-1})\\
 &=U^-t_1^{-1}k_1^{-1}(k_2')^{-1}\\
 &=U^-t't_2^{-1}k_2^{-1}(k_1')^{-1}t^{-1}\\
 &=U^-t't_2^{-1}(k_1')^{-1}t^{-1}
\end{align*}
Thus, $\sigma_{r_1,+}(u_1)\dr_1 = h_1't't_2^{-1}(k_1')^{-1}t^{-1}$. On the other hand, we compute that 
\begin{align*}
    (\tau_{r_2,+}(u_2)\dr_2)^{-1}\in \dr_2^{-1}u_2^{-1}U^- &= k_1'g_2^{-1}u_+^{-1}\dr_2^{-1}U^- \\
    & = k_1'g_2^{-1}\dr_2^{-1}U^-.
\end{align*}
Note that these arguments are independent of isogeny class of $G$. We have
$$(\sigma_{r_1,+}(u_1)\dot r_1)^\flat = (h_1't't_2^{-1}(k_1')^{-1}t^{-1})^\flat,\enspace (\tau_{r_2,+}(u_2)\dr_2)^{-1,\flat}\in  (k_1'g_2^{-1}\dr_2^{-1})^\flat (U^-)^\flat$$
In particular, all the torus elements lie in $\tilde T_{>0}$.

Suppose that $t = \prod_{i\in I}\omega^\vee_i(a^{-1}_i)$. Write $t^\natural = \prod_{i\in I}\tilde{\alpha}^\vee_{i^\natural}(a_i)$. It follows from our construction of $\tilde G$ that
$$t^{\flat,-1}t^\natural(k_1')^\flat t^{\natural,-1}t^{\flat}= (k_1')^\flat,$$
and then
\begin{align*}
\phi_{(r_1,r_2)}^\flat(u_1,u_2;\underline{a}) &\in (h_1't't_2^{-1}(k_1')^{-1}t^{-1})^\flat \left(\prod_{i\in I}y_{i^\natural}(a_i^{-1})\right)t^\natural(k_1'g_2^{-1}\dr_2^{-1})^\flat (U^-)^\flat\dot w_0^\flat\tilde B^+/\tilde B^+
\\
&\in (h_1't't_2^{-1})^\flat\tU^-_{c_{I^\natural},>0}(U^+_{w_2r_2,>0})^\flat\dr_2^{-1} (U^-)^\flat\dot w_0^\flat\tilde B^+/\tilde B^+ \\
&\subseteq (U_{r_1,>0}^-)^\flat\tU^-_{c_{I^\natural},>0}(U^+_{w_2r_2,>0})^\flat(\dr_2^{-1}\dot w_0)^\flat(B^+)^\flat.
\end{align*}

Similarly, 
we have
$$\phi^\sharp_{(r_1,r_2)}(u_1,u_2;\underline{a})\in (U^-_{r_2,>0})^\sharp\tU^-_{c_{I^\natural},>0}(U^+_{w_1^{-1}r_1,>0})^\sharp(\dot r_1^{-1}\dot w_0)^\sharp\tilde B^+/\tilde B^+.$$
By construction, we conclude that $\phi_{(r_1,r_2)}(U^-_{w_1,>0}T_{>0}U^+_{w_2,>0})$ is contained in 
\begin{align*}
&(U^-_{r_1,>0})^\flat(U^-_{r_2,>0})^\sharp\tU_{c_{I^\natural},>0}^-(U^+_{w_1^{-1}r_1,>0}\dr_1^{-1}\dot w_0)^\sharp(U^+_{w_2r_2,>0}\dr_2^{-1}\dot w_0)^\flat\tB^+/\tB^+ \\
=&\tilde{\CB}_{\varphi((I;1,w_1^{-1})\diamond(I;1,w_2^{-1})),\varphi((\emptyset;r_1,r_2)\diamond(\emptyset;r_1,r_2)),>0}.
\end{align*}
By \Cref{compatible_strata}, $\phi_{(r_1,r_2)}$ restricts to an isomorphism 
$$B^+\dot w_1B^+\cap B^-\dot w_2B^-\cap\dr_1U^-TU^+\dr_2^{-1}\simeq \mathring{\tilde{\CB}}_{\varphi((I;1,w_1^{-1})\diamond(I;1,w_2^{-1})),\varphi((\emptyset;r_1,r_2)\diamond(\emptyset;r_1,r_2))}.$$
On either side, the totally nonnegative cell is a connected component of the respective variety, and we have the desired equality. 
\end{proof}

\end{proposition}

\begin{proof}[Proof of \Cref{positive_atlas_for_compactfication}]
The two conditions in the definition of a positive atlas follow from \Cref{inside_open_compactification} and \Cref{image_positive}, where the latter is applied with $w_1=w_2=w_0$. 
\end{proof}

\subsection{Revisit $PGL_2$}
In \Cref{atlas_example}, we write down explicitly the atlas maps for $\overline{PGL_2}$. Now, we demonstrate that these maps preserve total positivity.

An element in $\mathcal O_{(1,1)}\cap G$ can be written as $g=y_1(a)\omega_1^\vee(b^{-1})x_1(c)$.
For $ac\ne0$, its image is
\begin{align*}
\phi_{(1,1)}(g)
&=\tilde y_{1^\natural}(b^{-1})\tilde\alpha_{1^\natural}^\vee(b)
  \tilde x_{1^\flat}(c)\tilde x_{1^\sharp}(a)
  \dot{\tilde s}_{1^\sharp}\dot{\tilde s}_{1^\flat}
  \tilde B^+/\tilde B^+\\
&=\tilde y_{1^\natural}(b^{-1})
  \tilde y_{1^\flat}(bc^{-1})
  \tilde y_{1^\sharp}(ba^{-1})\tilde B^+/\tilde B^+.
\end{align*}
The condition for either side to be totally positive is $a,b,c>0$.

An element in $\mathcal O_{(s_1,1)}\cap G$ can be written as
\[
g=x_1(a)\dot s_1\omega_1^\vee(b^{-1})x_1(c)
 =\begin{pmatrix}
 ab^{-1}&ab^{-1}c-1\\
 b^{-1}&b^{-1}c
 \end{pmatrix}\in PGL_2.
\]
For $a(ac-b)\ne0$, its image is
\begin{align*}
\phi_{(s_1,1)}(g)
&=\tilde x_{1^\flat}(a)\dot{\tilde s}_{1^\flat}
  \tilde y_{1^\natural}(b^{-1})\tilde\alpha_{1^\natural}^\vee(b)
  \tilde x_{1^\flat}(c)\dot{\tilde s}_{1^\flat}
  \tilde B^+/\tilde B^+\\
&=\tilde y_{1^\flat}(a^{-1})
  \tilde y_{1^\natural}(ab^{-1})
  \tilde y_{1^\flat}\!\left(\frac{b}{a(ac-b)}\right)
  \tilde B^+/\tilde B^+.
\end{align*}
The condition for either side to be totally positive is $a,b,ac-b>0$.

The calculations for $\phi_{(s_1,s_1)}$ and $\phi_{(1,s_1)}$ are similar.

\section{Combinatorial properties of the face poset}\label{section:combinatorics}
In this section, we study the face poset $\CQ$. Let $\hat\CQ = \CQ\sqcup\{\hat 0\}$ the augmented poset with a unique minimum adjoined. We show that $\hat\CQ$ is thin and CL-shellable.

\subsection{Properties of posets}\label{sec:comb}
Let $(P,\leq)$ be a poset. The covering relation is denoted by $\lessdot$.
For $x,y\in P$ with $x\leq y$, a maximal chain from $x$ to $y$ is a finite sequence
$x=z_0\lessdot z_1\lessdot\cdots\lessdot z_n=y$.
In general, a maximal chain may not exist.

The poset $P$ is \emph{pure} if, for any $x\leq y$, maximal chains from
$x$ to $y$ exist and all have the same length. A pure poset is
\emph{graded} if it contains a unique minimum and a unique maximum.
A pure poset is \emph{thin} if every interval of length $2$ has exactly
$4$ elements.

The \emph{order complex} $\Delta_{\mathrm{ord}}(P)$ of a finite poset
$P$ is the simplicial complex whose vertices are the elements of $P$
and whose simplices are its chains. A finite pure poset is
\emph{shellable} if its order complex is shellable: its maximal faces
can be ordered $F_1,\dots,F_n$ so that
$F_k\cap\bigl(\bigcup_{i<k}F_i\bigr)$ is a nonempty union of maximal
proper faces of $F_k$ for $k=2,\dots,n$.

\subsection{Thinness of $\hat\CQ$}
Recall that $\hat\CQ = \CQ\sqcup\{\hat 0\}$, where $\hat 0$ is an adjoined minimum. For $\alpha\diamond\beta\in\CQ$, let $\text{rank}(\alpha\diamond\beta) = \dim [\alpha\diamond\beta]+1$, and $\text{rank}(\hat 0) = 0$. It follows from \Cref{combinatorial_atlas} that $\hat\CQ$ is graded.

\begin{proposition}\label{thin}
The poset $\hat\CQ$ is thin.
\end{proposition}

\begin{proof}

By \Cref{combinatorial_atlas} and the thinness of $\tW$, every interval of length $2$ in $\mathcal Q$ contains exactly four elements. It remains to check that for $q\in\hat{\mathcal Q}$ with $\operatorname{rank}(q)=2$, there are exactly two rank-$1$ elements below $q$. Three cases occur:
\begin{itemize}
    \item $q = (\emptyset;x,y;v,y)$ with $v\le x$ and $\ell(x)=\ell(v)+1$: the two rank-$1$ elements are $\xi_{(v,y)}$ and $\xi_{(x,y)}$.
    \item $q = (\emptyset;x,y;x,w)$ with $y\le w$ and $\ell(w)=\ell(y)+1$: the two rank-$1$ elements are $\xi_{(x,y)}$ and $\xi_{(x,w)}$.
    \item $q = (\{i\};x,y;x,y)$ with $x,y\in W^{\{i\}}$: the two rank-$1$ elements are $\xi_{(x,y)}$ and $\xi_{(xs_i,ys_i)}$. 
\end{itemize}
In each case exactly two rank-$1$ elements are obtained. Hence $\hat{\mathcal Q}$ is thin. 
\end{proof}

\subsection{Shellability of $\hat\CQ$}
We prove shellability by constructing a CL-labeling of the entire
augmented poset.
\subsubsection{CL-shellable}
Let $P$ be a graded poset, and denote by $\hat 0$ (resp. $\hat 1$) the unique minimal (resp. maximal) element. We consider the set
$$E^*(P) = \{(m,x,y)|m\text{ a maximal chain of } [\hat 0,\hat 1],x,y\in m,x\lessdot y\}.$$
Let $\Lambda$ be a totally ordered set, a \emph{chain-edge labeling} of $P$ is a map $\lambda:E^*(P)\to\Lambda$ that satisfies the following
condition: If two maximal chains of $[\hat 0,\hat 1]$ coincide along their first $d$ edges, then they have the same
labels on these edges. In other words, if $m = (\hat 0 = x_0\lessdot x_1\lessdot\cdots\lessdot x_n =\hat 1)$ and $m' = (\hat 0 = x_0'\lessdot x_1'\lessdot\cdots\lessdot x_n' = \hat 1)$, where $x_i = x_i'$ for $i=0,1,\cdots,d$, then
$\lambda(m,x_{i-1},x_i) = \lambda(m',x_{i-1}',x_i')$ for $i=1,\cdots,d$. Given a chain-edge labeling $\lambda$, for every maximal chain $m = (\hat 0 = x_0\lessdot x_1\lessdot\cdots\lessdot x_n = \hat 1)$, we can assign a tuple of labels $(\lambda(m,x_0,x_1),\lambda(m,x_1,x_2),\cdots,\lambda(m,x_{n-1},x_n))$. A chain $m$ is called increasing if $\lambda(m,x_{i-1},x_i)<\lambda(m,x_i,x_{i+1})$ for $i=1,2,\cdots,n-1$.

For $x,y\in P$ and $x\leq y$, and $r$ a maximal chain of $[\hat 0,x]$, the pair $([x,y],r)$ is called a \emph{rooted interval} with $r$ the \emph{root}. We denote it by $[x,y]_r$. 

Consider a rooted interval $[x,y]_r$, and let $c$ be a maximal chain of $[x,y]$. Let $m$ be a maximal chain of $[\hat 0,\hat 1]$ extending $(r,c)$, and we have the label sequence of $m$. Let $\lambda_r(c)$ be the subsequence contributed by $c$, and by definition of chain-edge labeling, this label sequence only depends on $r$ and $c$, but does not depend on the extension $m$. We say that $c$ is an increasing chain if $\lambda_r(c)$ is an increasing sequence. Given a rooted interval, we can order its maximal chains by lexicographic order of their label sequence.

A chain-edge labeling is said to be a \emph{CL-labeling} if the following two conditions hold:
\begin{enumerate}
    \item for every rooted interval $[x,y]_r$, there is a unique increasing chain $c$;
    \item the label sequence $\lambda_r(c)$ is lexicographically minimal among all the label sequences of maximal chains in $[x,y]_r$.
\end{enumerate}
If a graded poset $P$ admits a CL-labeling, we say that $P$ is \emph{CL-shellable}. Bj\"orner and Wachs \cite{BW82} showed that if a graded poset is CL-shellable, then it is shellable. 

\begin{remark}
If the dual poset $P^*$ of $P$ is CL-shellable, we say that $P$ is dual CL-shellable. Note that $\triangle_{\text{ord}}(P)\simeq\triangle_{\text{ord}}(P^*)$. Therefore, $P$ being dual CL-shellable also implies that $P$ is shellable.
\end{remark}

\begin{remark} \label{EL-labeling}
There is a similar notion called EL-labeling which is automatically CL. The key difference is that, in an EL-labeling, a label only depends on the covering relation $x\lessdot y$ but not on the choice of a maximal chain. Therefore, one does not have to fix a root $r$ to talk about the label sequence of a maximal chain $c$ of an interval $[x,y]$. 
\end{remark}

\subsubsection{Reflection order}
Consider a Coxeter system $(W,I)$ together with the Bruhat order $\leq$. Let $T$ be the set of all reflections in $W$. A \emph{reflection order} is a total order $\prec$ on $T$ such that for every $s,t\in T$, we have either $t\prec sts\prec tstst\prec\cdots $ or $s\prec tst\prec ststs\prec\cdots$. 

For any interval $[w_1,w_2]$ of $(W,\leq)$, Dyer \cite[\S3]{Dy92} constructed a labeling as follows: if $v_1\lessdot v_2$, then $v_1^{-1}v_2\in T$, and for any maximal chain $m$ of $[w_1,w_2]$ that contains $v_1$ and $v_2$, define $\lambda_W(m,v_1,v_2) = v_1^{-1}v_2$, and this is a CL-labeling. With the same labeling, $[w_1,w_2]$ is also dual CL-shellable.

\begin{remark}
In fact, the labeling does not depend on $m$, and as in \Cref{EL-labeling}, this is an EL-labeling.
\end{remark}

\subsubsection{A CL-labeling on $\hat\CQ$}
The idea of the construction of a CL-labeling is as follows. Let $\min(\CQ)$ be the set of minimal elements in $\CQ$. For a chain-edge triple $(m,\hat 0, \xi)$ where $\xi\in\min(\CQ)$, we put $\lambda(m,\hat 0,\xi) = (\xi,\eta_\xi)$. To define other labels, we view $\CQ^*$ as a convex subset of $\tilde W$ as in \Cref{combinatorial_atlas}. For a tuple $(m,q_1,q_2)$, where $m$ is a maximal chain in $\hat\CQ$ passing through $\xi\in\min(\CQ)$, and $q_1,q_2\in m\cap\CQ$ such that $q_1\lessdot_{\CQ} q_2$, we put the label $\lambda(m,q_1,q_2) = (\xi,\varphi(q_2)\varphi(q_1)^{-1})$. 

It remains to compare the labels. We will first define a total order $\triangleleft_{\min(\CQ)}$ on $\min(\CQ)$. For each $\xi\in\min(\CQ)$, we choose a reflection order $\prec_\xi$ in the set $\tilde T$ of all reflections in $\tilde W$, and extend it to a total order on $\tilde T\sqcup\{\eta_\xi\}$ in a certain way. Then, we order the labels lexicographically, namely,
$$(\xi,\tilde t)\prec(\xi',\tilde t')\text{ if }\xi\triangleleft_{\min(\CQ)}\xi',\text{ or }\xi =\xi'\text{ and }\tilde t\prec_\xi\tilde t'.$$
Here the second factor may be the extra label $\eta_\xi$ or $\eta_{\xi'}$. Note that we never have to compare the second factor unless the two labels have the same first factor, so this total order is well-defined.

For a maximal chain $m$ of $\hat\CQ$ passing through $\xi_{(r_1,r_2)}\in\min(\CQ)$, we put $\lambda(m,\hat 0,\xi_{(r_1,r_2)}) = (\xi_{(r_1,r_2)},\eta_{(r_1,r_2)})$. For a maximal chain $m$ of $\hat\CQ$ passing through $\xi_{(r_1,r_2)}$, $q_1,q_2\in m\cap\CQ$ such that $q_1\lessdot_{\CQ} q_2$, we put $\lambda(m,q_1,q_2) = (\xi_{(r_1,r_2)},\varphi(q_2)\varphi(q_1)^{-1})$.

We refine the Bruhat order on $W$ to be a total order $<_W$, and define the total order $\triangleleft_{\min(\CQ)}$ as the lexicographic order on $(W,<_W)^2 = \min\{\CQ\}$.

To define the reflection order $\prec_{(r_1,r_2)}$, we will make use of the following lemma by Dyer \cite[\S2.11]{Dy93}.

\begin{lemma} \label{initial_section}
Let $(W,I)$ be a Coxeter system with length function $\ell$, and $T$ the set of reflections. For $w\in W$, let
$$N(w) = \{t\in T|\ell(tw) < \ell(w)\}.$$
There is a reflection order $\prec$ on $T$ such that $t_1\prec t_2$ for $t_1\in N(w)$ and $t_2\notin N(w)$.
\end{lemma}

For $r\in W$, denote $D(r):=\{i\in I|rs_i<r\}$ the right descent set of $r$. Recall that
$$\varphi(\xi_{(r_1,r_2)}) = r_1^\flat r_2^\sharp c_{(D(r_1))^\natural}c_{(I\setminus D(r_1))^\natural}(r_1^{-1}w_0)^\sharp(r_2^{-1}w_0)^\flat.$$
Let $\prec_{(r_1,r_2)}$ be a total order on $\tilde T\sqcup\{\eta_{(r_1,r_2)}\}$, which restricts to a reflection order on $\tilde T$, such that
$$\tilde t_1\prec_{(r_1,r_2)}\eta_{(r_1,r_2)}\prec_{(r_1,r_2)}\tilde t_2,\text{ for }\tilde t_1\in N(r_1^\flat r_2^\sharp c_{(D(r_1))^\natural}),\tilde t_2\in\tilde T\setminus N(r_1^\flat r_2^\sharp c_{(D(r_1))^\natural}).$$
By \Cref{initial_section}, such a total order exists.

Finally, we order the labels lexicographically, namely, 
$$(\xi_{(r_1,r_2)},\tilde t)\prec(\xi_{(r_1',r_2')},\tilde t')\text{ if }\xi_{(r_1,r_2)}\triangleleft_{\min(\CQ)}\xi_{(r_1',r_2')},\text{ or }\xi_{(r_1,r_2)} =\xi_{(r_1',r_2')}\text{ and }\tilde t\prec_{(r_1,r_2)}\tilde t'.$$
Here the second factor may be the extra label $\eta_{(r_1,r_2)}$ or $\eta_{(r_1',r_2')}$. Note that we never have to compare the second factor unless the two labels have the same first factor, so this total order is well-defined.

This defines a chain-edge labeling $\lambda$. We will use \Cref{atoms} and the following two lemmas, which are concerned with the rank-$2$ elements in $\hat\CQ$, to prove that it is a CL-labeling. Recall that $\hat\CQ$ is thin by \Cref{thin}.

\begin{lemma} \label{diamond_labels}
Let $q\in\hat\CQ$, $\xi_{(r_1,r_2)},\xi_{(r_1',r_2')}\in\min(\CQ)$, and $\xi_{(r_1,r_2)}\lessdot_\CQ q$, $\xi_{(r_1',r_2')}\lessdot_{\CQ}q$. Suppose that $\xi_{(r_1,r_2)}\triangleleft_{\min(\CQ)}\xi_{(r_1',r_2')}$. Then we have $$\varphi(q)\varphi(\xi_{(r_1,r_2)})^{-1}\notin N(r_1^\flat r_2^\sharp c_{(D(r_1))^\natural}),\enspace \varphi(q)\varphi(\xi_{(r_1',r_2')})^{-1}\in N((r_1')^\flat (r_2')^\sharp c_{(D(r_1'))^\natural}).$$

\begin{proof}
Let $A_{(r_1,r_2)}^1 = r_1^\flat r_2^\sharp c_{(D(r_1))^\natural}$, $A_{(r_1,r_2)}^2 = c_{(I\setminus D(r_1))^\natural}(r_1^{-1}w_0)^\sharp(r_2^{-1}w_0)^\flat$. Then $\varphi(\xi_{(r_1,r_2)}) = A^1_{(r_1,r_2)}A^2_{(r_1,r_2)}$ and $\ell(\varphi(\xi_{(r_1,r_2)})) = \ell(A_{(r_1,r_2)}^1)+\ell(A_{(r_1,r_2)}^2)$, and as a result, we have
$$N(A_{(r_1,r_2)}^1)\bigcap A_{(r_1,r_2)}^1N(A_{(r_1,r_2)}^2)(A^1_{(r_1,r_2)})^{-1} = \emptyset.$$

Similar to \Cref{thin}, there are three cases:
\begin{itemize}
   \item $q = (\emptyset;x,y;v,y)$ with $v\le x$ and $\ell(x)=\ell(v)+1$. In this case $r_1 = v$, $r_1' = x$, $r_2 = r_2' = y$. We have
   $$\varphi(q)\varphi(\xi_{(r_1,r_2)})^{-1} = v^\flat y^\sharp c_{I^\natural}(x^{-1}v)^\sharp c_{I^\natural}(y^{-1})^\sharp (v^{-1})^\flat\in A_{(r_1,r_2)}^1N(A_{(r_1,r_2)}^2)(A^1_{(r_1,r_2)})^{-1};$$
   $$\varphi(q)\varphi(\xi_{(r_1',r_2')})^{-1} = (vx^{-1})^\flat\in N(A_{(r_1',r_2')}^1).$$
   Hence the conclusion.
   
    \item $q = (\emptyset;x,y;x,w)$ with $y\le w$ and $\ell(w)=\ell(y)+1$. In this case $r_1 = r_1' = x$, $r_2 = y$, $r_2' = w$. A similar analysis as in the first case applies.
    
    \item $q = (\{i\};x,y;x,y)$ with $i\notin D(x)\cup D(y)$. In this case $r_1 = x$, $r_2 = y$ and $r_1' = xs_i$, $r_2' = ys_i$. We have
    $$\varphi(q)\varphi(\xi_{(r_1,r_2)})^{-1} = x^\flat y^\sharp s_{i^\natural}(y^{-1})^\sharp (x^{-1})^\flat\notin N(A_{(r_1,r_2)}^1);$$
   $$\varphi(q)\varphi(\xi_{(r_1',r_2')})^{-1} = (xs_i)^\flat (ys_i)^\sharp s_{i^\natural}(s_iy^{-1})^\sharp (s_ix^{-1})^\flat\in N(A_{(r_1',r_2')}^1).$$
\end{itemize}
This completes the proof.
\end{proof}

\end{lemma}

\begin{lemma} \label{decreasing_label}
Let $q = (J;x,y;v,w)\in\CQ$ and $\xi_{(r_1,r_2)}\in\min(\CQ)$ with  $\xi_{(r_1,r_2)}\leq_\CQ q$ and $\xi_{(r_1,r_2)}\neq q$. If $(r_1,r_2)\neq(v,y^J)$, then there exist $\xi_{(r_1',r_2')}\in\min(\CQ)$ and $q'\in\CQ$, such that
$$\xi_{(r_1',r_2')}\triangleleft_{\min(\CQ)}\xi_{(r_1,r_2)},\enspace \xi_{(r_1',r_2')}\lessdot_\CQ q'\leq_\CQ q,\enspace \xi_{(r_1,r_2)}\lessdot_\CQ q'.$$

\begin{proof}
Since $\xi_{(r_1,r_2)}\leq_{\CQ}q$, there exist $z_1,z_2\in W_J$, such that
$vz_2\leq r_1\leq xz_1$, and $yz_1\leq r_2\leq wz_2$. We consider three cases.

\begin{itemize}
    \item $r_1\neq vz_2$. Take $r_1'$ such that $vz_2\leq r_1'\lessdot r_1$, and let $q'=(\emptyset;r_1,r_2;r_1',r_2)$. We have
    $$vz_2\leq r_1'\leq r_1\leq xz_1,\enspace yz_1\leq r_2\leq wz_2,$$
    so $\xi_{(r_1',r_2)}\leq_\CQ q'\leq_\CQ q$ and $\xi_{(r_1',r_2)}\triangleleft_{\min(\CQ)}\xi_{(r_1,r_2)}$. 
    
    \item $r_1 = vz_2$ and $z_2\neq e$. Take $i\in D(z_2)$. Then $r_1s_i\leq r_1$. Since $z_2\in W_J$, we have $i\in J$. Suppose that $r_2\leq r_2s_i$. Then $r_2\leq wz_2s_i$, and we have
    $$vz_2s_i=r_1s_i\leq r_1\leq xz_1,\enspace yz_1\leq r_2\leq wz_2s_i.$$
    Take $q' = (\emptyset;r_1,r_2;r_1s_i;r_2)$ and then $\xi_{(r_1,r_2)}\lessdot_\CQ q'$, $\xi_{(r_1s_i,r_2)}\lessdot_{\CQ} q'\leq_\CQ q$ and $\xi_{(r_1s_i,r_2)}\triangleleft_{\min(\CQ)}\xi_{(r_1,r_2)}$.

    Suppose instead that $r_2s_i\leq r_2$. Put $r_1'=r_1s_i$, $r_2' = r_2s_i$. Let $q' = (\{i\};r_1',r_2';r_1',r_2')$. Then, $\xi_{(r_1',r_2')}\lessdot_{\CQ}q'$, $\xi_{(r_1,r_2)}\lessdot_{\CQ}q'$, and $\xi_{(r_1',r_2')}\triangleleft_{\min(\CQ)}\xi_{(r_1,r_2)}$. It remains to show that $q'\leq_\CQ q$, which follows from
    $$r_1's_i = r_1\leq xz_1,\enspace yz_1\leq r_2 = r_2's_i,\enspace vz_2s_i\leq r_1',\enspace r_2'\leq wz_2s_i.$$
    
    \item $v = r_1$. Since $(r_1,r_2)\neq (v,y^J)$, we have $r_2\neq y^J$. Take $r_2'$ such that $y^J\leq r_2'\lessdot r_2$, and let $q' = (\emptyset;r_1,r_2';r_1,r_2)$. Then $\xi_{(r_1,r_2)}\lessdot q'$, $\xi_{(r_1,r_2')}\lessdot q'$, and $\xi_{(r_1,r_2')}\triangleleft_{\min(\CQ)}\xi_{(r_1,r_2)}$. It remains to show that $q'\leq_{\CQ}q$, which we obtain from
    $$v = r_1\leq xy_J^{-1},\enspace yy_J^{-1} = y^J\leq r_2'\leq r_2\leq w.$$
\end{itemize}
This completes the proof.
\end{proof}

\end{lemma}

We are ready to show that the construction above gives a CL-labeling.

\begin{theorem} \label{CL-shellable}
The chain-edge labeling $\lambda$ on $\hat\CQ$ is CL. As a consequence, $\hat\CQ$ is shellable.

\begin{proof}
We first consider a rooted interval $[q_1,q_2]_r$ where $q_1\neq\hat 0$. By \Cref{combinatorial_atlas}, this is isomorphic to a Bruhat interval. By our construction, the labeling is CL from the result of Dyer \cite[\S3]{Dy92}.

Now, we consider an interval $[\hat 0,q]$ where $q = (J;x,y;v,w)$. Here, the root is trivial. By our definition of $\triangleleft_{\min(\CQ)}$, the lexicographically minimal chain passes through
$\xi_{(v,y^J)}$. Since the labeling is CL over the interval $[\xi_{(v,y^J)},q]$, the lexicographical minimal chain 
$$c = (\xi_{(v,y^J)}\lessdot_\CQ q_1\lessdot_{\CQ}q_2\lessdot_\CQ\cdots\lessdot_\CQ q)$$ 
is increasing, and the first label in this chain is $(\xi_{(v,y^J)},\varphi(q_1)\varphi(\xi_{(v,y^J)})^{-1})$, by our definition of $\prec_{(v,y^J)}$ and \Cref{diamond_labels}, we have
$$(\xi_{(v,y^J)},\eta_{(v,y^J)})\prec(\xi_{(v,y^J)},\varphi(q_1)\varphi(\xi_{(v,y^J)})^{-1}).$$
Therefore the maximal chain $(\hat 0,c)$ is lexicographically minimal and increasing. Moreover, there is only one increasing chain that passes through $\xi_{(v,y^J)}$.

Now we suppose that a maximal chain
$$(\hat 0\lessdot_{\hat\CQ}\xi_{(r_1,r_2)}\lessdot_{\hat\CQ} q_1'\lessdot_{\hat\CQ}q_2'\lessdot_{\hat\CQ}\cdots\lessdot_{\hat\CQ} q)$$
is increasing but $(r_1,r_2)\neq (v,y^J)$. Since $\lambda$ defines a CL-labeling on $[\xi_{(r_1,r_2)},q]$, we know that 
$$c' =(\xi_{(r_1,r_2)}\lessdot_{\hat\CQ} q_1'\lessdot_{\hat\CQ}q_2'\lessdot_{\hat\CQ}\cdots\lessdot_{\hat\CQ} q)$$
is lexicographically minimal. However, by \Cref{decreasing_label}, there is $q''_1\in\CQ$ and $\xi_{(r_1',r_2')}\in\min(\CQ)$ such that $\xi_{(r_1,r_2)}\lessdot_{\CQ}q''_1\leq_{\CQ} q$, $\xi_{(r_1',r_2')}\lessdot_\CQ q''_1$ and $\xi_{(r_1',r_2')}\triangleleft_{\min(\CQ)}\xi_{(r_1,r_2)}$. By \Cref{diamond_labels}, we have $\varphi(q_1'')\varphi(\xi_{(r_1,r_2)})^{-1}\prec_{(r_1,r_2)}\eta_{(r_1,r_2)}$. Since $c'$ is lexicographically minimal, we have 
$$\varphi(q_1')\varphi(\xi_{(r_1,r_2)})^{-1}\preceq_{(r_1,r_2)}\varphi(q_1'')\varphi(\xi_{(r_1,r_2)})^{-1}\prec_{(r_1,r_2)}\eta_{(r_1,r_2)},$$
and thus, $(\xi_{(r_1,r_2)},\varphi(q_1')\varphi(\xi_{(r_1,r_2)})^{-1})\prec (\xi_{(r_1,r_2)},\eta_{(r_1,r_2)})$. This is a contradiction, and a maximal chain passing through $\xi_{(r_1,r_2)}$ with $(r_1,r_2)\neq(v,y^J)$ can never be increasing. This completes the proof.
\end{proof}

\end{theorem}

\section{The regularity theorem}\label{section:regularity}

We now combine the positive atlas with the combinatorial properties of the face poset.
The positive atlas supplies a manifold structure on the closure of each cell. Shellability and thinness of the face poset
identify their boundaries as spheres. A topological criterion
based on the generalized Poincar\'e conjecture then proves the
following theorem.

\begin{theorem}\label{regular_remarkable}
The decomposition
\[
\overline{G_{\ge0}}
 =\bigsqcup_{\alpha\diamond\beta\in\CQ}
     [\alpha\diamond\beta]_{>0}
\]
is a remarkable polyhedral
space with the regularity property. 
\end{theorem}

\subsection{Regular CW complexes}
In this subsection, we recall some topological results. The first result
\cite{Bjo84} connects shellability to the topology of regular CW complexes.

\begin{proposition}\label{shellable_regular}
Let $X$ be a regular CW complex with face poset $P$. If
$P\sqcup\{\hat0,\hat1\}$, obtained by adjoining a minimum and a
maximum, is graded, thin, and shellable, then $X$ is homeomorphic to
a sphere of dimension
$\operatorname{rank}(P\sqcup\{\hat0,\hat1\})-2$.
\end{proposition}

The following theorem follows from the generalized Poincar\'e
conjecture and Brown's collar theorem; see
\cite[Theorem~2.2]{BH24} and the references therein.

\begin{theorem}\label{thm:poincare}
Let $X$ be a compact $n$-dimensional topological manifold with
boundary such that
\begin{enumerate}
    \item $\partial X$ is homeomorphic to an $(n-1)$-dimensional sphere;
    \item $X\setminus\partial X$ is homeomorphic to an
    $n$-dimensional open ball.
\end{enumerate}
Then $X$ is homeomorphic to an $n$-dimensional closed ball $D^n$.
\end{theorem}

We use the above results to give a criterion for the regularity property of a remarkable polyhedral space. 

\begin{proposition}\label{boundary_regular}
Let $(X=\bigsqcup X_\gamma,X_{\ge0})$ be a remarkable polyhedral
space, and write $X_{\gamma,>0}=X_\gamma\cap X_{\ge0}$.
Let $P$ be the closure poset of its \emph{nonempty} nonnegative
cells, and put $\hat P=P\sqcup\{\hat0\}$. Suppose that
\begin{enumerate}
    \item for every $\gamma\in P$, the interval $[\hat0,\gamma]$
    is finite, thin, and shellable;
    \item for every $\gamma\in P$, the Hausdorff closure
    $\overline{X_{\gamma,>0}}$ is a compact topological manifold
    with boundary
    \[
    \partial\overline{X_{\gamma,>0}}
      =\bigsqcup_{\gamma'\in P,\,\gamma'<\gamma}X_{\gamma',>0}.
    \]
\end{enumerate}
Then $X_{\ge0}$ has the regularity property.
\end{proposition}

\begin{proof}
We induct on $n=\dim X_{\gamma,>0}$. If $n=0$, the cell and its closure are a single point.

Now we consider general $\gamma$. By induction hypothesis, $\partial \overline{X_{\gamma,>0}}$ is a regular CW complex, with face poset $[\hat 0,\gamma]\setminus\{\hat 0,\gamma\}$. It then follows from condition (1) and \Cref{shellable_regular} that $\partial\overline{X_{\gamma,>0}}$ is homeomorphic to a sphere of codimension $1$ in $\overline{X_{\gamma,>0}}$. Finally, it follows from \Cref{thm:poincare} that $\overline{X_{\gamma,>0}}$ is homeomorphic to a closed ball with interior $X_{\gamma,>0}$. This completes the proof.
\end{proof}

\subsection{Proof of \Cref{regular_remarkable}}
It follows from \Cref{positive_atlas_for_compactfication} that $(\overline{G},\overline{G_{\ge0}})$ admits a positive atlas, and thus by \Cref{positive_atlas} is a remarkable polyhedral space. Moreover, the closure of each cell $[\alpha\diamond\beta]_{>0}$ is a topological manifold with boundary
$$\partial\overline{[\alpha\diamond\beta]_{>0}} = \bigsqcup_{\alpha'\diamond\beta'<_{\mathcal Q}\alpha\diamond\beta}[\alpha'\diamond\beta']_{>0}.$$
Since $\overline{G}$ is projective, the closure $\overline{[\alpha\diamond\beta]_{>0}}$ is compact. The augmented face poset is graded, thin, and shellable by \Cref{combinatorial_atlas}, \Cref{thin} and \Cref{CL-shellable}. It then follows from \Cref{boundary_regular} that $\overline{G_{\ge0}}$ has the regularity property.

\section{Examples of types \texorpdfstring{$A_1$ and $A_2$}{A1 and A2}} \label{section:examples}
In this section, we present concrete examples of the totally nonnegative part of the wonderful compactification of $PGL_2$ and $PGL_3$.

\subsection{Wonderful compactification of projective linear groups} \label{linear_compactification}
We consider the projective linear group $G = PGL_n$. Let $G_{\text{sc}} = SL_n$ be the simply connected cover of $G$. Denote $I = \{1,2,\cdots, n-1\}$. For $i\in I$, we denote $\omega_i$ the corresponding fundamental weight of $G_{\text{sc}}$, and $V_{\omega_i}$ the corresponding irreducible highest weight module. Then $V_{\omega_1}\simeq\mathbb R^n$ is the natural representation of $SL_n$, and $V_{\omega_k} = \wedge^k\mathbb R^n$ is the $k$th exterior power of the natural representation. 

For $g\in PGL_n(\mathbb R)$, choose a matrix representative
$A\in GL_n(\mathbb R)$. The map
\[
PGL_n\longrightarrow\prod_{i\in I}\mathbb P(\operatorname{End}(V_{\omega_i})),
\qquad
[A]\longmapsto([A],[\wedge^2 A],\ldots,[\wedge^{n-1}A])
\]
is independent of this choice. According to \cite{DP83}, the wonderful compactification $\overline G$ is the Zariski closure of the image.

\subsection{Type $A_1$}\label{sec:rank_one_details}
Recall from \Cref{ex:rank_one} that $\overline{PGL_2}\simeq\mathbb P^3$, and 
$$\overline{PGL_{2,\ge0}} = \{[a:b:c:d]\in\mathbb P^3\big|a,b,c,d,ad-bc\ge0\}.$$
We now describe its cells and closure relations.

The unique open $G\times G$-orbit is isomorphic to $PGL_2$ and its totally nonnegative part is the monoid $PGL_{2,\ge0}$. It consists of one $3$-cell where all minors are positive, two $2$-cells:
\begin{align*}
U^-_{s_1,>0}T_{>0} = \left\{[a:0:c:1]\in\mathbb P^3\big|a,c>0\right\}, \\
U_{s_1,>0}^+T_{>0} = \left\{[a:b:0:1]\in\mathbb P^3\big|a,b>0\right\},
\end{align*}
and one $1$-cell:
$$T_{>0}=\left\{[a:0:0:1]\in\mathbb P^3\big|a>0\right\}.$$

The unique closed $G\times G$-orbit is isomorphic to $\mathbb P^1\times\mathbb P^1$, and its totally nonnegative part is $\mathbb P^1_{\ge0}\times\mathbb P^1_{\ge0}$. It consists of one $2$-cell, four $1$-cells, and four $0$-cells. 

The Hasse diagram of the face poset and the regular CW complex are illustrated as follows.

\begin{figure}[htbp]
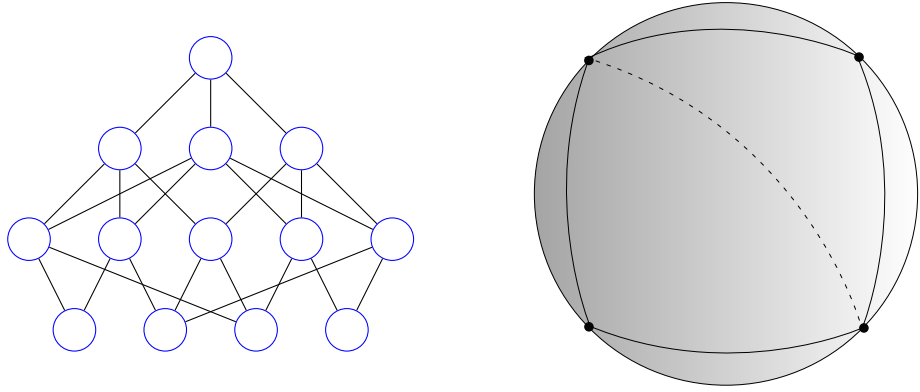

\begin{minipage}{0.45\linewidth}
\begin{tikzpicture}[scale=0.6,
        every node/.style={font=\footnotesize, circle, draw=blue, minimum size=16pt, inner sep=1pt}]
\node (1) at(4,8){};
\node (2) at(2,6){};
\node (3) at(4,6){};
\node (4) at(6,6){};
\node (5) at(0,4){};
\node (6) at(2,4){};
\node (7) at(4,4){};
\node (8) at(6,4){};
\node (9) at(8,4){};
\node (10) at(1,2){};
\node (11) at(3,2){};
\node (12) at(5,2){};
\node (13) at(7,2){};

\draw[-] (1)--(2);
\draw[-] (1)--(3);
\draw[-] (1)--(4);
\draw[-] (2)--(5);
\draw[-] (2)--(6);
\draw[-] (2)--(7);
\draw[-] (3)--(5);
\draw[-] (3)--(6);
\draw[-] (3)--(8);
\draw[-] (3)--(9);
\draw[-] (4)--(7);
\draw[-] (4)--(8);
\draw[-] (4)--(9);
\draw[-] (5)--(10);
\draw[-] (5)--(12);
\draw[-] (6)--(10);
\draw[-] (6)--(11);
\draw[-] (7)--(11);
\draw[-] (7)--(12);
\draw[-] (8)--(12);
\draw[-] (8)--(13);
\draw[-] (9)--(11);
\draw[-] (9)--(13);
\end{tikzpicture}
\end{minipage}
\hfill
\begin{minipage}{0.45\linewidth}
  \IfFileExists{sl2.pdf}{\includegraphics[width=2in]{sl2.pdf}}{
\IfFileExists{sl2.png}{\includegraphics[width=2in]{sl2.png}}{
\begin{tikzpicture}[scale=.85]
\fill[red!5] (0,0)--(3,0)--(3,3)--(0,3)--cycle;
\draw[red,thick] (0,0)--(3,0)--(3,3)--(0,3)--cycle;
\draw[red,thick] (0,0)--(3,3);
\foreach \p in {(0,0),(3,0),(3,3),(0,3)}\fill[red] \p circle (2pt);
\node[red] at (2,.7) {$2$-cell};
\node[red] at (1,2.3) {$2$-cell};
\node[red,align=center,font=\scriptsize] at (1.5,-.6)
{Boundary sphere: the exterior\\is the third $2$-cell};
\end{tikzpicture}
}}
\end{minipage}
\caption{Hasse diagram and cell complex}
\end{figure}

The boundary is a two-dimensional sphere consisting of a quadrilateral
and two triangles. The quadrilateral is
$\mathbb P^1_{\ge0}\times\mathbb P^1_{\ge0}$. The two triangles are
$\overline{U^-_{s_1,>0}T_{>0}}$ and
$\overline{U^+_{s_1,>0}T_{>0}}$, and they share the edge
$\overline{T_{>0}}$, whose endpoints are $[1:0:0:0]$ and $[0:0:0:1]$.

\subsection{Type $A_2$}
For $G=PGL_3$, both fundamental representations $V_{\omega_1}$ and $V_{\omega_2}$ have dimension $3$. We think of $\mathbb P^8$ as the projective space of $\text{Mat}_3$, the space of $3\times 3$ matrices. After re-numbering and re-scaling coordinates, the embedding in \S\ref{linear_compactification} can be written as
$$PGL_3\hookrightarrow\mathbb P^8\times\mathbb P^8,\enspace g\mapsto ([g],[g^{-1}]).$$
The Zariski closure is 
$$\overline{PGL_3} = PGL_3\bigsqcup\left\{([A],[B])\in\mathbb P^8\times\mathbb P^8\big|AB=BA = 0\right\}.$$
The $G\times G$-action on $\overline G$ can be described as
$$(g_1,g_2)\cdot ([A],[B]) = ([g_1Ag_2^{-1}],[g_2Bg_1^{-1}]).$$

Let $\text{Mat}_{3,>0}$ be the space of all totally positive $3\times 3$ matrices, $\text{Mat}_{3,\ge0}$ be the space of totally nonnegative but not necessarily invertible matrices. By Whitney's Theorem \cite{Whit52}, the Hausdorff closure of $\text{Mat}_{3,>0}$ in $\text{Mat}_3(\mathbb R)$ is $\text{Mat}_{3,\ge0}$.

We have that $\overline{PGL_{3,\ge0}}$ is
$$PGL_{3,\ge0}\bigsqcup\{([A],[B])\in\mathbb P^8\times\mathbb P^8|AB=BA=0,A\in\text{Mat}_{3,\ge0},\dot w_0^{-1}B\dot w_0\in\text{Mat}_{3,\ge0}\}.$$
Here, $\dot w_0 = \begin{pmatrix}
    0 & 0 & 1 \\
    0 & -1 & 0 \\
    1 & 0 & 0
\end{pmatrix}$.

Recall that, for $J\subseteq I = \{1,2\}$, we have a point $z_J\in\overline{G}$ such that
$$\text{Stab}_{G\times G}(z_J) = (U_{P_J^+},U_{P_J^-}) \{(l_1,l_2)\in L_J\times L_J|l_1l_2^{-1}\in Z(L_J)\}.$$
In this case, we may take
$$z_I = \left(\begin{bmatrix}
    1 & 0 & 0 \\
    0 & 1 & 0 \\
    0 & 0 & 1
\end{bmatrix},\begin{bmatrix}
    1 & 0 & 0 \\
    0 & 1 & 0 \\
    0 & 0 & 1
\end{bmatrix}\right),\enspace z_\emptyset = \left(\begin{bmatrix}
    1 & 0 & 0 \\
    0 & 0 & 0 \\
    0 & 0 & 0
\end{bmatrix},\begin{bmatrix}
    0 & 0 & 0 \\
    0 & 0 & 0 \\
    0 & 0 & 1
\end{bmatrix}\right),$$
$$z_{\{1\}} = \left(\begin{bmatrix}
    1 & 0 & 0 \\
    0 & 1 & 0 \\
    0 & 0 & 0
\end{bmatrix},\begin{bmatrix}
    0 & 0 & 0 \\
    0 & 0 & 0 \\
    0 & 0 & 1
\end{bmatrix}\right),\enspace z_{\{2\}} = \left(\begin{bmatrix}
    1 & 0 & 0 \\
    0 & 0 & 0 \\
    0 & 0 & 0
\end{bmatrix},\begin{bmatrix}
    0 & 0 & 0 \\
    0 & 1 & 0 \\
    0 & 0 & 1
\end{bmatrix}\right).$$
The open $G\times G$-orbit $(G\times G)\cdot z_I$ is isomorphic to $G = PGL_3$, and its nonnegative part is the totally nonnegative monoid $PGL_{3,\ge0}$, which consists of $36$ cells of various dimensions, indexed by pairs of elements in the symmetric group $S_3$. The closed $G\times G$-orbit $(G\times G)\cdot z_\emptyset$ is isomorphic to the product of two copies of the full flag variety of $PGL_3$, and its totally nonnegative part consists of $19^2 = 361$ cells of various dimensions, indexed by pairs of Bruhat intervals in $S_3$.

The $G\times G$-orbit $(G\times G)\cdot z_{\{1\}}$ can be described as 
$$\left\{([A],[B])\in\mathbb P^8\times\mathbb P^8\big| AB=BA=0,\text{rank}(A)=2,\text{rank}(B)=1\right\}.$$

Within this orbit, the totally nonnegative part is characterized by
representatives satisfying $A\in\operatorname{Mat}_{3,\ge0}$ and
$(-1)^{i+j}b_{ij}\ge0$ for $B=(b_{ij})$.
By the description of $\CQ$ in \S\ref{sec:wonderful_strata}, the cells
are indexed by two Bruhat intervals in $W^{\{1\}}$ and a pair of elements
in $W_{\{1\}}\simeq S_2$. The three-element chain $W^{\{1\}}$ has six
intervals, giving $6^2\cdot2^2=144$ cells. By symmetry, the totally nonnegative part of the other
intermediate $G \times G$-orbit also has $144$ cells.

Thus $\overline{PGL_{3,\ge0}}$ is an $8$-dimensional regular CW complex
with $36+361+144+144=685$ cells. The numbers of cells
in each dimension are
\[
\begin{array}{c|ccccccccc}
\text{dimension}&0&1&2&3&4&5&6&7&8\\ \hline
\text{number of cells}&36&114&173&166&112&56&21&6&1
\end{array}
\]
Their alternating sum is $1$, as expected for a closed ball.


\begin{thebibliography}{99}

\bibitem{BH21a} H. Bao and X. He, \emph{Flag manifolds over semifields}, Algebra \& Number Theory {\bf 15}(8), 2037--2069 (2021)

\bibitem{BH21b}
H. Bao and X. He, \emph{A Birkhoff-Bruhat Atlas for partial flag varieties}, Indag. Math. {\bf 32}, 1152--1173 (2021), special issue to the memory of T.A. Springer.

\bibitem{BH22}
H. Bao and X. He, \emph{Total positivity in twisted product of flag varieties}, arXiv:2211.11168.

\bibitem{BH24}
H. Bao and X. He, \emph{Product structure and regularity theorem for totally nonnegative flag varieties}, Invent. Math. {\bf 237}, no. 1, 1--47  (2024). 

\bibitem{BD94}Y. Billig and M. J. Dyer, \emph{Decompositions of Bruhat type for Kac--Moody groups}, Nova J. Algebra Geom.
{\bf 3}(1), 11--39 (1994).

\bibitem{Bjo80}A. Bj\"orner, \emph{Shellable and Cohen-Macaulay partially ordered sets}, Trans. Amer. Math. Soc. {\bf 260}, 159--183 (1980).

\bibitem{Bjo84}A. Bj\"orner, \emph{Posets, regular CW complexes and Bruhat order}, Eur. J. Comb. {\bf 5}, 7--16 (1984)

\bibitem{BW82}A. Bj\"orner and M. Wachs, \emph{Bruhat order of Coxeter groups and shellability}, Adv. Math. {\bf 43} 87--100 (1982).


\bibitem{CD03}Y. Chen and M. Dyer, \emph{On the combinatorics of $B\times B$-orbits on group compactifications}, J. Algebra. {\bf 263}(2), 278-293 (2003). 

\bibitem{DP83}C. De Concini and C. Procesi, \emph{Complete symmetric varieties}, Invariant theory (Montecatini,
1982), Lecture Notes in Mathematics {\bf 996}, 1--44 (1983).

\bibitem{Dy92}M. J. Dyer, \emph{Hecke algebras and shellings of Bruhat intervals. II. Twisted Bruhat orders}, In: Kazhdan-Lusztig Theory and Related Topics, Chicago, IL, 1989. Contemp. Math., vol. 139, pp. 141--165. Am. Math. Soc., Providence (1992)

\bibitem{Dy93}M. J. Dyer, \emph{Hecke algebras and shellings of Bruhat intervals}, Compositio Mathematica {\bf 89}, 91--115 (1993).



\bibitem{GKL22}P. Galashin, S. Karp, and T. Lam, \emph{Regularity theorem for totally nonnegative flag varieties}, J. Am. Math. Soc. {\bf 35}(2), 513--579 (2022).

\bibitem{He04}X. He, \emph{Total positivity in the De Concini-Procesi Compactification}, Represent. Theory {\bf 8}, 52-71 (2004).

\bibitem{He05}X. He,
\emph{Some subvarieties of the De Concini-Procesi compactification}, Massachusetts Institute of Technology (2005).

\bibitem{He07}X. He, \emph{Minimal length elements in some double cosets of Coxeter groups}, Adv. in Math.,
{\bf 215}, 469-503 (2007).

\bibitem{HL11}X. He and J. Lu, \emph{On Intersections of Certain Partitions of a group compactification}, International Mathematics Research Notices, {\bf 2011}(11), 2534--2564 (2011).



\bibitem{HKL}
X. He, A. Knutson, and J.-H. Lu, unpublished notes (2013).

\bibitem{HSX25} X. He, C. Simpson, and K. Xie, \emph{Total positivity for matroid Schubert varieties}. Sel. Math. New Ser. {\bf 31}, 83 (2025).

\bibitem{HX26}X. He and K. Xie, \emph{Total positivity in twisted flag varieties}, arXiv:2602.09350.


\bibitem{Her14}P. Hersh, \emph{Regular cell complexes in total positivity}, Invent. Math. {\bf 197}, 57--114 (2014)

\bibitem{KL79}
D. Kazhdan and G. Lusztig, \emph{Representations of Coxeter groups and Hecke algebras}, Invent. Math. {\bf 53}, 165--184 (1979).

\bibitem{LLM02}
V. Lakshmibai, P. Littelmann, and P. Magyar, \emph{Standard Monomial Theory for Bott--Samelson Varieties}, Compositio Mathematica {\bf 130}, 293--318 (2002).

\bibitem{L94} G. Lusztig, 
\emph{Total positivity in reductive groups}, Lie theory and geometry, Progress in Mathematics, {\bf 123}, 531-568 (1994).

\bibitem{L20} G. Lusztig, \emph{Positive structures in Lie theory}, ICCM Not. {\bf 8}(1), 50--54 (2020).

\bibitem{MR04}B. Marsh and K. Rietsch, \emph{Parametrizations of flag varieties}, Represent. Theory {\bf 8}, 212--242 (2004).

\bibitem{Ri99}
K. Rietsch, \emph{An algebraic cell decomposition of the nonnegative part of a flag variety}, 
J. Algebra {\bf 213}(1), 144--154 (1999). 


\bibitem{Ri06}K. Rietsch, \emph{Closure relations for totally nonnegative cells in $G/P$}, Math. Res. Lett. {\bf 13}, no.5-6, 775--786 (2006).  

\bibitem{Spr07}T. A. Springer, \emph{Some results on compactifications of semisimple groups}, Proceedings of the International Congress of Mathematicians Madrid {\bf 2}, 1337--1348 (2007).

\bibitem{Whit52}A. Whitney, \emph{A reduction theorem for totally positive matrices}, J. d'Analyse Math. {\bf 2}, 88--92 (1952).



\end{thebibliography}
\end{document}